\pdfoutput=1

\documentclass[12pt]{amsart}
\usepackage{epsfig,amsmath}
\usepackage{graphics}

\begin{document}

\numberwithin{equation}{section}

\theoremstyle{plain}
\newtheorem{thm}[subsection]{Theorem}
\newtheorem{lem}[subsection]{Lemma}
\newtheorem{cor}[subsection]{Corollary}
\newtheorem{prop}[subsection]{Proposition}
\newtheorem{obs}[subsection]{Observation}

\theoremstyle{definition}
\newtheorem{definition}[subsection]{Definition}

\theoremstyle{remark}
\newtheorem{remark}[subsection]{Remark}
\newtheorem{eg}[subsection]{Example}
\newtheorem{egs}[subsection]{Examples}

\newcommand{\R}{\mathbf{R}}
\newcommand{\US}{\mathbf{S}}

\def\cpo{\textsf{cpo}}  
\def\cx{\mathbf{c}}
\def\ax{\mathbf{a}}

\def\a{\alpha}
\def\b{\beta}
\def\g{\gamma}
\def\G{\Gamma}
\def\k{\kappa}
\def\w{\omega}
\def\th{\theta}
\def\z{\zeta}
\def\eps{\varepsilon}

\def\bd{\partial}
\def\i{\mathbf{i}}
\def\ov{\mathcal{O}}
\def\ovb{\underline{\ov}}
\def\tube{\mathcal{T}}
\def\Peps{P_{\tau,b}(\eps)}
\def\theps{\th_{\eps}}
\def\mag#1{\left|#1\right|}
\def\half{{\textstyle{\frac{1}{2}}}}

\def\C{\mathbf{C}}
\def\O{\mathbf{0}}
\def\teps{t_{\eps}}

\def\ds{\displaystyle}

\hyphenation{para-metrize para-metriz-ation repara-metrize repara-met-riz-ation}

\long\def\symbolfootnote[#1]#2{
	\begingroup
	\def\thefootnote{\fnsymbol{footnote}}\footnote[#1]{#2}
	\endgroup
}
	
\parskip 5pt
\parindent 0pt
\baselineskip 16pt


\author{Bruce Solomon}
\address{Math Department, Indiana University, Bloomington IN 47405}
\email{solomon@indiana.edu}
\urladdr{mypage.iu.edu/$\sim$solomon}

\keywords{Quadric surface, oval, central symmetry, skewloop}

\begin{abstract}
Say that a surface in $\,S\subset\R^{3}\,$ has the \emph{central plane oval property, 
or \cpo}, if 
\medskip

\begin{itemize}
	\item
	$S\,$ meets some affine plane transversally along an oval, and 
	\smallskip
		
	\item
	Every such transverse plane oval on $\,S\,$ has central symmetry. 
\end{itemize}
\smallskip

\noindent
We show that a complete, connected $\,C^{2}\,$ surface with \cpo\ must either be 
a cylinder over a central oval, or else \emph{quadric}.
\smallskip

\noindent
We apply this to deduce that a complete $\,C^{2}\,$ surface containing a 
transverse plane oval but no skewloop, must be cylindrical or quadric.
\end{abstract}

\title{surfaces with central convex cross-sections}
\maketitle

\section{Introduction and overview}\label{sec:intro}

Call a set in a euclidean space \emph{central} if it has symmetry with respect to 
reflection through a point---its \emph{center}. Call an embedded plane loop an 
\emph{oval} if its curvature never vanishes.

\begin{figure}[h]
	\begin{center}
		\includegraphics[height=2cm]{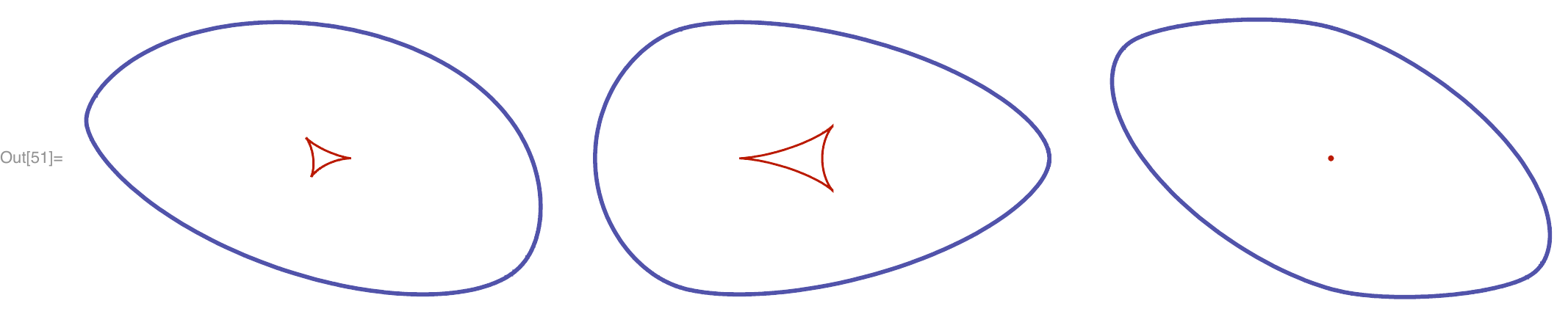}
		\caption{Ovals and their centrices (see \S\ref{ssec:ctx}). Only the rightmost oval is central.}
	\end{center}
\end{figure}
If we erect a cylinder over a central oval in $\,\R^{3}\,$, its transverse planar cross-sections,
whenever compact, will be central ovals too.

The same goes for \emph{quadrics}---level-sets of a quadratic polynomials on $\,\R^{3}\,$: Their transverse planar cross-sections, when compact, are always ellipses, 
which are certainly central ovals.

We show here that these two kinds of examples provide the \emph{only} complete 
$\,C^{2}\,$ surfaces in $\,\R^{3}\,$ whose planar ovals are all central.  
We will call this the \emph{central plane oval} property and abbreviate it 
by \emph{\cpo}:

\goodbreak
\begin{definition}[\textbf{\cpo}]\label{def:cpo}
	A $\,C^{2}$-immersed surface $\,S\subset\R^{3}\,$ has the \emph{central plane oval
	property}, or \emph{\cpo}, if
	
	\begin{itemize}
		
		\item
		$S\,$ intersects \emph{at least one} affine plane transversally along an 
		{oval}, and
		
		\item
		\emph{Every} time $\,S\,$ intersects an affine plane transversally along an oval,
		that oval is \emph{central}
	\end{itemize}

\end{definition}

Given this terminology, we can state our main result as follows:

\textbf{Theorem \ref{thm:main}} (Main Theorem)
	\emph{A complete, connected $\,C^{2}$-immersed surface in $\,\R^{3}\,$ with  \cpo\ is 
	either a cylinder, or quadric.}

This result complements a fundamentally local fact about \emph{convex} surfaces
proven long ago by W.~Blaschke in \cite{b}:

\begin{prop}[{\cite[1918]{b}}]\label{prop:blaschke}
	Suppose every plane transverse, and nearly tangent to, a smooth \emph{convex} surface 
	$\,S\subset\R^{3}\,$ cuts $\,S\,$ along a central loop. Then $\,S\,$ is quadric. 
\end{prop}

Though it resembles---and helped to inspire---our Main Theorem above, Blaschke's result 
seems much easier to prove, for the simple reason that convex surfaces lie on one side of their 
tangent planes. By pushing such a plane slightly into the surface, one always cuts it in a small convex 
loop. Blaschke merely observed that when all such loops are \emph{central}, one can Taylor-expand 
the surface as a graph over any tangent plane with no cubic term. This annihilates the Pick invariant 
on the surface, making it quadric.

Contrastingly, our Theorem allows some, or even all of the surface, to have negative Gauss
curvature. In a negatively curved region, one \emph{never} finds arbitrarily small 
planar ovals, and this totally blocks any direct generalization of Blaschke's argument---as 
he himself laments in \cite{b}.

We thus find it necessary to approach Theorem \ref{thm:main} using a global, multi-stage argument 
that ultimately rests on the rotationally symmetric case. We published the latter result in \cite{sor}:

\begin{prop}[\cite{sor}]\label{prop:sor}
	Let $\,M\,$ be a surface of revolution.
	If $\,M\,$ intersects every plane nearly perpendicular to its axis 
	in a central set, then $\,M\,$ is quadric.
\end{prop}

The fundamental problem we must solve to get from this basic result to our Main Theorem
boils down to the case of a general ``tube''. For suppose an immersed surface $\,M\,$ meets some 
plane transversally along an oval as our definition of \cpo\ requires. Then some neighborhood, in $\,M\,$,
of that oval embeds into $\,\R^{3}\,$ as a roughly cylindrical tube with \cpo.
Such tubes turn out to form the critical test case for our work. 
To explain further, we need some precise language. 

Let $\,I:=(-1,1)\,$ denote the open unit interval.

\begin{definition}[\textbf{Transversely convex tube}]\label{def:tct}
	Suppose $\,X:\US^{1}\times I\to\R^{3}\,$ is an embedding of the form
	\[
		X(\th,z) := \bigl(\cx(z) + \g(z;\th),\,z\bigr)\,,
	\]
	where $\,\cx:I\to\R^{2}\,$ and $\,\g:I\times\US^{1}\to\R^{2}\,$ are $\,C^{2}\,$, 
	and for each fixed $\,z\in I\,$, the map
	$\,\g(z;\cdot):\US^{1}\to\R^{2}\,$ parametrizes a plane oval having its centroid at the origin.
	
	\noindent
	A \textbf{transversely convex tube} is any embedded annulus that, after an
	affine isomorphism, can be parametrized in this way. We call $\,\cx\,$ its  
	\textbf{central curve}. When studying a transversely convex tube, we lose no
	generality by assuming it to lie in the slab $\,|z|<1\,$ as parametrized above,
	and we will routinely do so without further comment.
		
	\noindent
	Discarding the central curve $\,\cx\,$ of a transversely convex tube $\,\tube\,$ in standard
	position, we get the \textbf{rectification} $\,\tube\,$, denoted $\,\tube^{*}\,$, and given
	by the image of 
	\[
		X^{*}(\th,z) :=\bigl(\g(z;\th),\,z\bigr)
	\]
	(Figure \ref{fig2}). Finally, we say that $\,\tube^{*}\,$ \textbf{splits} when
	\[
		\g(z;\th) = r(z)\,\g(\th)
	\]
	for some fixed oval $\,\g:\US^{1}\to\R^{2}\,$, and some positive scaling function 
	$\,r:I\to(0,\infty)\,$. Note that a split tube is a surface of revolution precisely when$\,\g\,$
	parametrizes an origin-centered circle.
	
	\begin{figure}[t]
		\begin{center}
			\includegraphics[height=4cm]{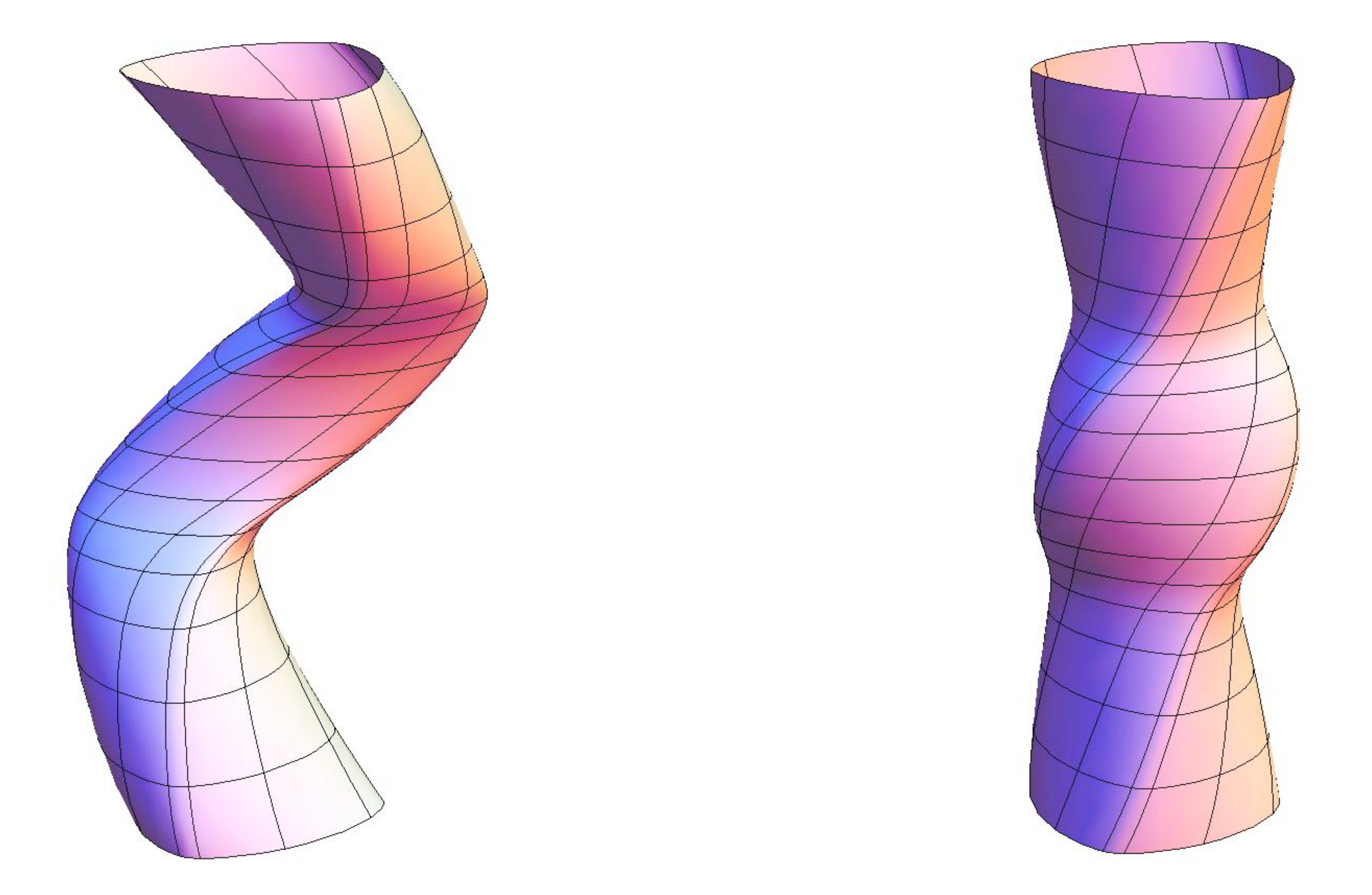}
		\caption{A transversely convex tube $\,\tube\,$ (left) and its rectification $\,\tube^{*}\,$ (right).}
		\label{fig2}
		\end{center}
	\end{figure}

\end{definition}

In these terms, we reach a key analytical juncture in our work when we prove the following 
technical result:

\textbf{Proposition \ref{lem:split}.} (Splitting Lemma)
	\emph{If a transversely convex tube $\,\tube\,$ in standard position has \cpo,
	then its rectification $\,\tube^{*}\,$ splits.}

Simple as this statement is, proving it was the most challenging part of our work. Much of 
the effort goes toward deriving a pair of partial differential equations satisfied by the function 
$\,h:\US^{1}\times I\to\R\,$ which, for each $\,z\in I\,$, yields the {support function} 
$\,h(z,\cdot)\,$ of the oval $\,\g(z;\cdot)\,$---the height-$z$ cross-section of the rectified tube 
$\,\tube^{*}\,$. These PDE's form the conclusion of Proposition \ref{prop:pdes}, and 
we devote most of \S\ref{sec:tubes} to their derivation. Our approach has a variational
flavor that we sketch out at the beginning of \S\ref{sec:tubes}.

We then obtain our Splitting Lemma by playing these PDE's off against each other.
Specifically, we use information gleaned from the second equation to rewrite the first
as an equation for the \emph{square} of $\,h\,$. We then notice a first integral for that
equation, and finally prove splitting with the help of ODE techniques in which the second
equation again plays a role.

Once we have Splitting, we return again to the first PDE from Proposition \ref{prop:pdes}, 
where we can now separate variables. This yields independent elementary ODE's for the 
horizontal and vertical behavior of our tube. Solving these, we reach the key geometric turning 
point of our work: We find that the possibilities for a tube with \cpo\ branch in two directions: 

\textbf{Proposition \ref{prop:branch}.} (Cylinder/Quadric)
	\textit{
	Suppose $\,\tube\,$ is a transversely convex tube with \cpo.
	Then its rectification $\,\tube^{*}\,$ is either
	\begin{itemize}
		\item[(i)] The cylinder over a central oval, or
		\smallskip
		\item[(ii)] Affinely congruent to a surface of revolution.
	\end{itemize}
	}

By Proposition \ref{prop:sor}, however, surfaces of revolution having \cpo\ are already quadric.
So we now see that, insofar as tubes go, it remains only to eliminate the rectification step. 
We do this in \S\ref{sec:axis} by proving

\textbf{Proposition \ref{prop:axis}.} (Axis lemma)
	\emph{Suppose $\,\tube\,$ is a transversally convex tube with \cpo. 
	Then its central curve is affine, 
	so that $\,\tube\,$ is affinely congruent to its rectification $\,\tube^{*}\,$.}

Together, the Cylinder/Quadric Proposition, Axis Lemma, and rotationally invariant case (Proposition 
\ref{prop:sor}) combine to show that a transversely convex tube with \cpo\ is either cylindrical 
or quadric. In other words, we have a ``tubular'' version of our Main Theorem:

\textbf{Proposition \ref{prop:collar}} (Collar Theorem)
	A transversely convex tube with \cpo\ is either cylindrical or quadric.

In \S\ref{sec:main}, we start with this fact, and show that it ``propagates,'' using an open/closed
argument, to any complete $\,C^{2}\,$ immersion with \cpo. This proves our Main Theorem 
\ref{thm:main}, and the argument is not difficult. For as we mentioned above, any surface $\,M\,$
with \cpo\ contains an annular subset that embeds in $\,\R^{3}\,$ as a transversely convex tube.
Our Collar Theorem now makes that tube either cylindrical or quadric. But the boundaries of such 
a tube, in either case, are again transverse central ovals. So they too have annular neighborhoods 
that embed as transversely convex tubes. Roughly speaking, this pushes the boundaries of the tube 
a little further out along $\,M\,$, and by completeness, the process terminates only when the tube 
engulfs all of $\,M\,$.

We conclude in \S\ref{sec:skew}, with an application that first motivated us toward the Main Theorem 
here: We extend the main result from our earlier paper with M.~Ghomi on skewloops \cite{gs}. 

A \emph{skewloop} is a smoothly immersed loop in $\,\R^{3}\,$ with no pair of distinct parallel tangent 
lines. In \cite{gs}, we showed that \emph{when a complete $C^{2}$-immersed surface in $\,\R^{3}\,$ 
has a point of positive curvature, it contains a skewloop if and only if it is not quadric.} 
We required the positive curvature 
assumption because our proof cited Proposition \ref{prop:blaschke} above (Blaschke's theorem) in an
essential way.  The Main Theorem here lets us bypass that result, eliminating the positive
curvature assumption in favor of one that holds for many surfaces with \emph{no} positive curvature:
the existence of a single transverse planar oval. We thus obtain

\textbf{Theorem \ref{thm:noskew}.}
	Suppose a $C^{2}$-immersed surface $\,M\subset\R^{3}\,$ crosses some plane transversally 
	along an oval. Then exactly one of the following holds:	
	\begin{itemize}
		
		\item[(i)] $S\,$ contains a skewloop.
		\smallskip
		
		\item[(ii)] $S\,$ is the cylinder over an oval.
		\smallskip
		
		\item[(iii)] $S\,$ is a non-cylindrical quadric.
	\end{itemize} 

For instance, this result characterizes the tube (i.e.~one-sheeted) hyperboloids as \emph{the only
negatively curved surfaces that contain a transverse plane oval, but no skewloop}.

We now proceed from the overview above to the details of our paper, starting with
some preliminary facts about ovals.

\goodbreak

\section{Oval and Centrix}\label{sec:ovals}

Recall that by an \emph{oval} in the plane, we mean an embedded, \emph{strictly} convex $\,C^{2}\,$ 
loop, and that a \emph{central oval} has central symmetry---symmetry with respect to reflection
through a point called its \emph{center}. 


\begin{definition}[Support parametrization/support function]\label{def:sptfn}
	A map $\,\g:\US^{1}\to\R^{2}\,$ \emph{support parametrizes} an oval $\,\ov\subset\R^{2}\,$ 
	if and only if it satisfies 
	\begin{equation}\label{eqn:spt}
		\g'(\th) = \left|\g'(\th)\right|\i\,e^{\i\th}
		\quad\text{for all $\,\th\in\R\,$}\,.
	\end{equation}
	Here we have identified $\,\C\approx\R^{2}\,$, and we regard
	$\,2\pi$-periodic maps $\,\R\to\R^{2}\,$ as maps from $\,\US^{1}\,$ to $\,\R^{2}\,$, 
	in the obvious ways. We use these identifications without 
	further comment below.
	
	Notice that (\ref{eqn:spt}) characterizes parametrization by the inverse of the 
	outer unit normal. This is a diffeomorphism $\,\ov\to\US^{1}\,$ on any $\,C^{2}\,$ oval $\,\ov$,
	a fact that yields both existence and uniqueness of the support parametrization.
	
	By an easy exercise, the \emph{support function} $\,h:\R\to\R\,$, given by 
	\begin{equation}\label{eqn:sptfn1}
		h(\th) := \sup_{p\in\ov}p\cdot e^{\i\th}\,,
	\end{equation}
	determines $\,\g\,$ via the formula
	\begin{equation}\label{eqn:sptfn2}
		\g(\th) = \left(h(\th) + \i\,h'(\th)\right)\,e^{\i\th}\,.
	\end{equation}
	%
\end{definition}

	Note that when we rotate an oval $\,\ov\,$ counterclockwise through an angle $\,\phi\,$ 
	about the origin, (\ref{eqn:sptfn1}) shifts its support function right by $\,\phi\,$: 
	\begin{equation}\label{eqn:hrot}
		h(\th)\mapsto h(\th-\phi)\,.
	\end{equation}
	Elementary calculations using (\ref{eqn:sptfn2}) further show that the
	support parametrization makes speed and curvature reciprocal to each other:
	\begin{equation}\label{eqn:speed}
		\left|\g'(\th)\right| = h(\th) + h''(\th)
		\qquad\text{and}\qquad
		\k(\th) = \frac{1}{h(\th)+h''(\th)}\,.
	\end{equation}
	In particular, strict convexity of an oval ensures that its support parametrization 
	\emph{immerses} the circle into $\,\R^{2}\,$.
	
	We eventually want to show that the cross-sectional ovals of a tube with \cpo\
	are circular up to affine isomorphism---ellipses. We will do so by invoking
	

\begin{obs}\label{prop:H'''}
	An oval is an origin-centered \textbf{ellipse} 
	if and only if its support function $\,h\,$ satisfies 
	\begin{equation*}
		\left(h^{2}\right)''' + 4\left(h^{2}\right)' = 0\,.
	\end{equation*}
\end{obs}

\begin{proof}
	We may parametrize any origin-centered ellipse by
	\[
		\a(t) = A\,e^{\i\,t}
	\]
	for some symmetric invertible matrix $\,A_{2\times2}\,$. In that case, 
	(\ref{eqn:sptfn1}) computes its support function as
	\[
		h(\th)
		=
		\sup_{t}\,A\,e^{\i\,t}\cdot e^{\i\th}\
		=
		\sup_{t}\,e^{\i\,t}\cdot Ae^{\i\th}\,.
	\]
	This supremum here clearly occurs when
	\[
		e^{\i\,t}=\frac{Ae^{\i\th}}{\left|Ae^{\i\th}\right|}\,,
	\]
	which instantly yields $\,h(\th) = \left|Ae^{\i\th}\right|\,$. Familiar trig identities
	then make it easy to deduce
	\begin{equation}\label{eqn:soe}
		h^{2}(\th) = a\cos(2\th + b) + c>0\,,
	\end{equation}
	for some constants $\,a,b\,$ and $\,c\,$, with $\,|a|<c\,$, and the positive solutions of 
	$\,f'''+4f'=0\,$ are precisely the functions given by (\ref{eqn:soe}).
	\end{proof}
	
Geometrically, (\ref{eqn:soe}) characterizes the support function of an ellipse with
major and minor axes  $\,\sqrt{c\pm a}\,$.


\subsection{The centrix.}\label{ssec:ctx}
We measure the failure of an oval to be centrally symmetric by examining the auxilliary 
curve that we call its \emph{centrix}:

\begin{definition}[Centrix]\label{def:centrix}
	Given an oval $\,\ov\subset\R^{2}\,$ and a unit vector 
	$\,e^{\i\th}\in\US^{1}\,$, there exist exactly two points on $\,\ov\,$
	with tangent lines perpendicular to $\,e^{\i\th}\,$. We call the line
	segment joining these two points the \emph{$\th$-diameter} of $\,\ov\,$.
	Denoting its midpoint by $\,\cx(\th)\,$, we then call the image of the resulting
	map $\,\cx:\US^{1}\to\R^{2}\,$ the \emph{centrix} of $\,\ov\,$. 
\end{definition}

\begin{figure}[h] 
	\centering
	\includegraphics[height=5cm]{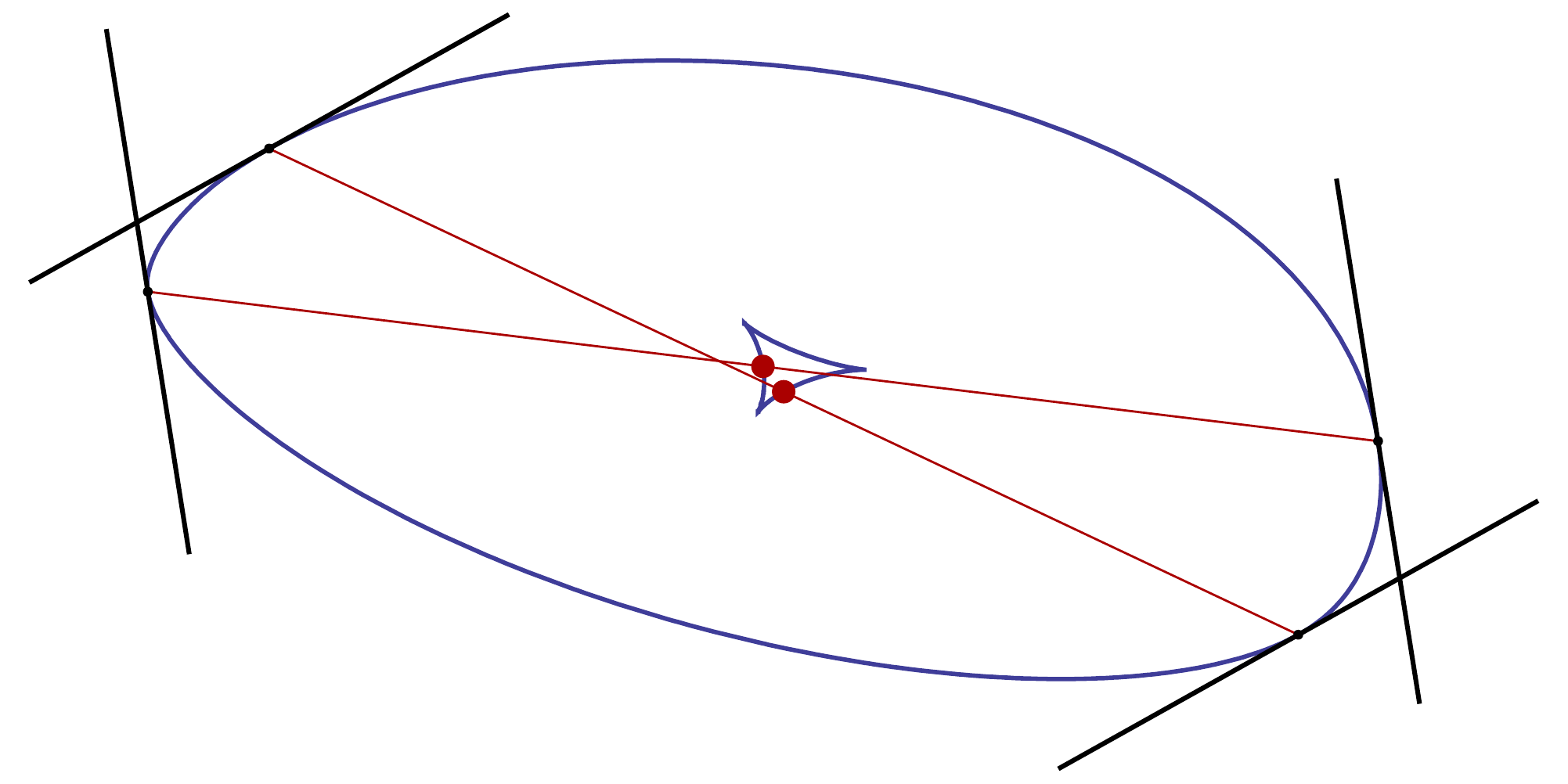}
	\caption{Midpoints of diameters trace out the centrix.}\label{fig:diameters}
\end{figure} 

\begin{definition}[Even/odd]\label{def:eo}
Given the support parametrization $\,\g\,$ of an oval $\,\ov\,$, we call the maps
\begin{equation*}
	\half\,\left(\g(\th) +\g(\th+\pi)\right)
	\quad\text{and}\quad
	\half\,\left(\g(\th)-\g(\th+\pi)\right)\,,
\end{equation*}
the \emph{even} and \emph{odd} parts of $\,\g\,$ respectively.
\end{definition}

\goodbreak
\begin{obs}\label{obs:centrix}	
	The centrix $\,\cx:\US^{1}\to\R^{2}\,$ of $\,\ov\,$ coincides with the even part of $\,\g\,$.
	It is a constant if and only if $\,\ov\,$ has central symmetry.
	In that case, the odd part of $\,\g\,$ support-parametrizes the origin-centered
	oval $\,\ov-\cx\,$.
\end{obs}

\begin{proof}
The defining condition for the support parametrization  (\ref{eqn:spt}) puts 
the endpoints of each $\th$-diameter on $\,\ov\,$ at $\,\g(\th)\,$ and 
$\,\g(\th+\pi)\,$. It follows immediately that the even part of $\,\g\,$ parametrizes
the centrix\nobreak $\,\cx\,$.

When $\,\cx(\th)\equiv \cx_{0}\in\R^{2}\,$, reflection through $\,\cx_{0}\,$ clearly
preserves $\,\ov\,$. 

Conversely, if reflection through some point $\,\cx_{0}\,$ preserves $\,\ov\,$, it---like any affine 
isomorphism---must preserve pairs of parallel lines. In particular, it will swap the endpoints of 
each $\th$-diameter, preserving their midpoints.
But reflection through $\,\cx_{0}\,$ preserves no other point. So central symmetry
means $\,\cx(\th)\equiv \cx_{0}\,$.

The even and odd parts of $\,\g\,$ always add back to $\,\g\,$.  So when $\,\ov\,$ is central, 
the odd part $\,\g^{*}\,$ clearly parametrizes $\,\ov-\cx\,$, whose center of 
symmetry obviously lies at the origin. In this case, we also have $\,(\g^{*})'(\th) = \g'(\th)\,$, 
a multiple of $\,\i\,e^{\i\th}\,$. It follows that (\ref{eqn:spt})
must hold for $\,\g^{*}\,$, which makes it a support parametrization.  
\end{proof}

\goodbreak

\section{Splitting}\label{sec:tubes}

In this section we tackle the technical key to our Main Theorem, establishing that
\cpo\ forces the support function of a transversely convex tube to split along purely horizontal 
and vertical factors. Our Splitting Lemma \ref{lem:split} states this precisely, and the
geometric consequence that makes it interesting, our Cylinder/Quadric Proposition
\ref{prop:branch}, then follows fairly easily.

To prepare for the Splitting Lemma, we need calculations that stretch over a number
of pages. We hope the following descriptive plan-of-attack will help the reader navigate
them with a clear sense of our intentions.

Our strategy is to focus on the families of ovals one gets by intersecting a transversely 
convex tube $\,\tube\,$ with planes tilted slightly away from the horizontal. Specifically,
given any $\,\eps\in\R\,$ and any unit-vector $\,\tau\in\US^{1}\,$, we consider the 
$\eps$-tilted plane given by
\begin{equation}\label{eqn:peps}
	\Peps
	:= 
	\left\{(p,z)\in\R^{2}\times\R:\ z = \eps\,(p\cdot\tau)+b\right\}.
\end{equation}
We call $\,\tau\,$ the \textbf{tilt-direction}, $\,b\,$ the $z$-intercept, and $\,\eps\,$ the 
\textbf{slope} of this plane. Fixing $\,\tau\in\US^{1}\,$ and $\,b\in(-1,1)\,$, we vary
the slope $\,\eps\,$ of this plane, and study the resulting intersections
with $\,\tube\,$  near $\,\eps=0\,$.

Since $\,\tube\,$ is transversely convex, it intersects \emph{horizontal} planes in 
$\,C^{2}\,$ ovals. By transversality, the cross-section $\,\Peps\cap\tube\,$
remains a $\,C^{2}\,$ oval for all sufficiently small $\,\eps\,$. When we assume that
$\,\tube\,$ has \cpo, these ovals all have central symmetry too. 

Our key idea is to study the \emph{centrices} of these cross-sections. The preservation of central symmetry makes them all singletons, by Observation \ref{obs:centrix}---they are independent
of the variable $\,\th\,$ along each oval. Differentiation with respect to $\,\th\,$ therefore yields 
a vanishing condition. By taking an initial $\eps$-derivative of this condition at 
$\,\eps=0\,$, we produce the two partial differential equations of Proposition \ref{prop:pdes}.
As explained in our introduction, these equations lead fairly directly to our Splitting Lemma.

We now work out the details of this program.

\subsection{The support map of $\,\tube\,$.}\label{ssec:Gamma}
As above, we let $\,\tube\,$ denote a transversely convex tube in standard position. 
By Definition \ref{def:tct} $\,\tube\,$ intersects the horizontal plane at any height
$\,b\in(-1,1)\,$ in an oval we shall call $\,\ov(b)\,$. Denote by $\,\nu:\tube\to\US^{1}\,$ 
the map that assigns to each point $\,p=(x,y,z)\in\tube\,$ the (horizontal) outer unit normal 
to $\,\ov(z)\,$ at $\,p\,$. Clearly, the map
\begin{equation*}
	\tube\to\US^{1}\times(-1,1)
	\quad\text{given by}\quad
	p\longmapsto \left(\nu(p),\ z(p)\right)\,.
\end{equation*}
is a diffeomorphism, whose inverse takes the form
\begin{equation}\label{eqn:dm}
	\left(e^{\i\th},\,z\right) \longmapsto \left(\G(\th,z),\,z\right)
\end{equation}
for some smooth map $\,\G:\US^{1}\times(-1,1)\to \R^{2}\,$. Indeed, 
$\,\G\,$ reparametrizes $\,\tube\,$, and for fixed $\,b\in(-1,1)\,$, it inverts the unit
normal map on $\,\ov(b)\,$. As mentioned following Definition \ref{def:sptfn}, this
means that $\,\G(\cdot,b)\,$ support-parametrizes $\,\ov(b)\,$, and for this reason,
we call it the \emph{support map} of the tube $\,\tube\,$.


\subsection{The height function $\z\,$}\label{ssec:implicit}
We now take an arbitrary intercept $\,-1<b<1\,$ and tilt direction $\,\tau\in\US^{1}\,$,
and regard them, for now, as fixed. 

Define the cross-section
\begin{equation*}
	\bar\ov(b,\eps):=\tube\cap\Peps\,,
\end{equation*}
and its image under the projection $\,(x,y,z)\stackrel{\pi}{\mapsto}(x,y)\,$,
\begin{equation*}
	\ov(b,\eps):= \pi\left(\bar\ov(b,\eps)\right)\,.
\end{equation*}
We abbreviate the \emph{horizontal} ($\eps=0$) cross-section by
\begin{equation*}
	\ov(b) := \bar\ov(b,0)\,,
\end{equation*}
and we will not hesitate to identify $\,\ov(b)\,$ with $\,\ov(b,0)\,$ too,
since the latter is clearly congruent to $\,\bar\ov(b,0)\,$.

As discussed above, the transverse convexity of $\,\tube\,$ ensures that  
$\,\bar\ov(b,\eps)\,$ is an oval for all sufficiently small $\,\eps\,$. When $\,\tube\,$
has \cpo, these tilted ovals will clearly have central symmetry as well, but we need 
not assume \cpo\ for our immediate goal here: We want to introduce and study the 
``height function'' $\,\z(\eps,\th)\,$ that lets us parametrize
$\,\ov(b,\eps)\,$ by the map (compare (\ref{eqn:dm}))
\begin{equation}\label{eqn:zeta0}
	\th\longmapsto \Bigl(\G\left(\th,\,\z(\eps,\th)\right),\ \z\left(\eps,\th\right)\Bigr)\,.
\end{equation}
The Implicit Function Theorem ensures the existence and $\,C^{2}\,$ smoothness
of $\,\z\,$. For suppose---informed by the characterization of $\,\Peps\,$ in 
(\ref{eqn:peps})---we define a map $\,G:\R\times\US^{1}\times(-1,1)\to\R\,$ via
\begin{equation}\label{eqn:G}
	G(\eps,\th,\z) := \z - b - \eps\,\tau\cdot\G(\th,\z)\,.
\end{equation}
Then $\,G\,$ inherits $\,C^{2}\,$ smoothness from $\,\G\,$, and
the pre-image of $\,\ov(b,\eps)\,$ in $\,\US^{1}\times(-1,1)\,$ under 
the parametrization of $\,\tube\,$ in (\ref{eqn:dm}) clearly solves
\begin{equation*}
	G(\eps, \th, \z) = 0\,.
\end{equation*}
On the \emph{horizontal} oval $\,\ov(b)\,$, we have $\,\z\equiv b\,$, so that trivially, 
\begin{equation*}
	G(0,\th,b)\equiv 0
	\quad\text{and}\quad
	\frac{\bd G}{\bd \z}\left(0,\theta,b\right)=1\ne 0
	\quad\text{for all $\,\theta\in\US^{1}\,$.}
\end{equation*}
The Implicit Function Theorem then provides a $\,\delta>0\,$, 
and a $\,C^{2}\,$ mapping $\,\z:(-\delta,\delta)\times\US^{1}\to\R\,$ that satisfies
\begin{equation}\label{eqn:zeta=z}
	\z(0,\th) \equiv b\quad\text{for all $\,\th\in\US^{1}\,$}\,,
\end{equation}
and
\begin{equation*}
	G\left(\eps,\th,\z(\eps,\theta)\right)\equiv 0
	\quad\text{for all $\,\theta\in\US^{1},\,|\eps|<\delta\,$.}
\end{equation*}

Written out using (\ref{eqn:G}), the latter equation becomes
\begin{equation}\label{eqn:zeta}
	\z(\eps, \th)
	=
	b+\eps\,\tau\cdot\G\left(\th,\,\z(\eps,\th)\right),
\end{equation}
which shows that, as hoped, (\ref{eqn:zeta0}) parametrizes $\,\ov(b,\eps)\,$. 

Now observe that the projection $\,(x,y,z)\stackrel{\pi}{\to}(x,y)\,$ induces an affine 
isomorphism $\,\Peps\approx\R^{2}\,$. Such maps preserve strict convexity, 
so that $\,\ov(b,\eps)\,$, and of course $\,\ov(b)\,$, are again ovals. 

For future reference, we note that affine isomorphisms also preserve central symmetry.
So when $\,\tube\,$ has \cpo, the projected oval $\,\ov(b,\eps)\,$ further inherits the
central symmetry that \cpo\ ascribes to $\,\bar\ov(b,\eps)\,$.

In any case, it will suffice henceforth to study the projected ovals $\,\ov(b,\eps)\,$ 
as it varies with $\,\eps\,$. In view of (\ref{eqn:zeta0}), we may clearly parametrize 
$\,\ov(b,\eps)\,$ by the immersion
\begin{equation}\label{eqn:G1}
	\th\longmapsto\G\left(\th,\,\z(\eps,\th)\right)\ .
\end{equation}

To analyze the initial variation of the centrix of $\,\ov(b,\eps)\,$, we will eventually 
requires following facts about the derivatives of $\,\z\,$. 
The reader will easiliy confirm them by differentiating 
(\ref{eqn:zeta}) implicitly, and using (\ref{eqn:zeta=z}):

\begin{obs}\label{obs:zeta'}
	We have
	\medskip
	%
	\begin{equation*}
		\frac{\bd\z}{\bd\eps}(0,\th)
		= 
		\tau\cdot
		\G\left(\th,\,b\right) 	
	\end{equation*}
	and
	\begin{equation*}
		\frac{\bd^{2}\z}{\bd\eps\,\bd\th}\left(0,\th\right)
		= 
		\tau\cdot \frac{\bd{\G}}{\bd\th}\left(\th,b\right)\ .
	\end{equation*}
\end{obs}


\subsection{The support-reparametrizing map $\,\theps\,$} 
Though (\ref{eqn:G1}) parame\-trizes $\,\ov(b,\eps)\,$, we 
want to study the \emph{centrix} of $\,\ov(b,\eps)\,$. Observation \ref{obs:centrix} 
offers a way to parametrize the centrix, but it derives from the \emph{support} parametrization 
of $\,\ov(b,\eps)\,$, not the one given by (\ref{eqn:G1}). The Proposition below details the
needed reparametrization, and its final conclusion yields a crucial input to our proof of the 
Splitting Lemma \ref{lem:split}.  Notation is as above.

\goodbreak
\begin{prop}\label{prop:theps}
	There exists a $\,\delta>0\,$ and a differentiable 1-para\-meter family of 
	diffeomorphisms
	\[
		\theps:\US^{1}\to\US^{1}\quad -\delta<\eps<\delta\,,
	\]
	such that the composition
	\[
		\G_{\eps}\circ \theps =  \G\left(\theps,\,\z(\eps,\theps)\right)
	\]
	support-parametrizes $\,\ov(b,\eps)\,$ for each $\,\eps\in(-\delta,\delta)\,$. The initial map
	$\,\th_{0}\,$ is the identity on $\,\US^{1}\,$, with initial $\eps$-derivative given by
	
	\[
		\frac{d\theps}{d\eps}\Big|_{\eps=0} 
		= 
		\left(\tau\cdot\i e^{\i\th}\right)\left(\frac{\bd\G}{\bd \z}(\th,b)\cdot e^{\i\th}\right)\,.
	\]
\end{prop}

\begin{proof}
The existence of $\,\theps\,$ is routine. For, $\,\G\left(\th,\z(\eps,\th)\right)\,$ parametrizes 
$\,\ov(b,\eps)\,$, and is $\,C^{2}\,$ in both $\,\th\,$ and $\,\eps\,$. This makes the unit outer normal 
$\,\nu_{\eps}(\th)\,$ on $\,\ov(b,\eps)\,$ continuously differentiable 
in both variables, while the strict convexity of $\,\ov(b,\eps)\,$ ensures that $\,\nu_{\eps}\,$ 
induces a diffeomorphism $\,\US^{1}\to\US^{1}\,$ that varies smoothly with $\,\eps\in(-\delta,\delta)\,$. 
By the Inverse Function Theorem, the inverse of this map varies smoothly in $\,\eps\,$ too. 
As noted after Definition \ref{def:sptfn}, however, the inverse of the outer normal on an oval gives its support parametrization. We therefore get the desired family of reparametrizing maps 
by setting $\,\theps:=(\nu_{\eps})^{-1}\,$ for each $\,|\eps|<\delta\,$.

Note too that by (\ref{eqn:zeta=z}), setting $\,\eps =0\,$ reduces $\,\G\left(\th,\z(\eps,\th)\right)\,$ 
to $\,\G(\th,b)\,$, which already support-parametrizes $\,\ov(b)\,$, by definition of $\,\G\,$. 
So $\,\th_{0}\,$ is the trivial reparametrization---the identity map---as claimed.

It remains to verify the stated formula for $\,\bd\theps/\bd\eps\,$ at $\,\eps=0\,$.
This requires some careful calculations.

Start by observing that since $\,\G_{\eps}\circ\theps\,$ support-parametrizes $\,\ov(b,\eps)\,$
when $\,|\eps|<\delta\,$. By (\ref{eqn:spt}), this makes its velocity at any input $\,\th\,$
a multiple of $\,\i e^{\i\th}\,$.
Hence
\begin{equation*}\label{eqn:star0}
	0\equiv e^{\i\th}\cdot \frac{\bd}{\bd\th}\left(\G_{\eps}\circ\theps\right)\,.
\end{equation*}
Use the chain rule to expand the derivative, abbreviating $\,\theps(\th)\,$ as simply
$\,\theps\,$, to rewrite this condition as
\begin{eqnarray*}
	0 &=&
	e^{\i\th}\cdot\frac{\bd}{\bd\th}\, \G\bigl(\theps,\, \z(\eps,\,\theps)\bigr)\\
	&&\\
	&=&
	e^{\i\th}\cdot
		\left[
			\frac{\bd\G}{\bd\th}\bigl(\theps,\z(\eps,\theps)\bigr) +
			\frac{\bd\G}{\bd\z}\bigl(\theps,\z(\eps,\theps)\bigr)
			\frac{\bd\z}{\bd\th}\left(\eps,\theps\right)
		\right]
		\frac{\bd\theps}{\bd \th}
	\end{eqnarray*}
Since $\,\theps\,$ is a diffeomorphism of $\,\US^{1}\,$, its derivative 
along the circle never vanishes. So we can divide out the final factor above
and conclude that {for all $\,|\eps|<\delta\,$, we have}

\begin{equation}\label{eqn:star}
	\frac{\bd\z}{\bd\th}\left(\eps,\theps\right)\,
	\frac{\bd\G}{\bd\z}\Bigl(\theps,\z\left(\eps,\theps\right)\Bigr)\cdot e^{\i\th}
	\ =\
	-\frac{\bd\G}{\bd\th}\Bigl(\theps,\z\left(\eps,\theps\right)\Bigr)\cdot e^{\i\th}\,.
\end{equation}
Regarding this as a characterization of $\,\theps\,$, we will differentiate implicitly with
respect to $\,\eps\,$, then set $\,\eps=0\,$ to verify the Proposition's final claim. 
To manage the task, we differentiate the two sides of (\ref{eqn:star}) 
separately before equating them to get our final conclusion.

\textit{Left side of (\ref{eqn:star}):} 
Differentiate the left-hand side of (\ref{eqn:star}). Because $\,\z(0,\th)\equiv b\,$, all pure 
$\th$-derivatives of $\,\z\,$ vanish at $\,\eps=0\,$, and we can rewrite the sole surviving summand using Observation \ref{obs:zeta'}:

\begin{eqnarray}\label{eqn:lhs'}
	\lefteqn{
		\frac{\bd}{\bd\eps}\Big|_{\eps=0}
		\left[
			\frac{\bd\z}{\bd\th}\left(\eps,\theps\right)\,
			\frac{\bd\G}{\bd\z}\Bigl(\theps,\z\left(\eps,\theps\right)\Bigr)\cdot e^{\i\th}
		\right]
	}\hspace{1.75in}\nonumber\\
	&&\nonumber\\
	&=&
	\frac{\bd^{2}\z}{\bd\th\,\bd\eps}\left(0,\th\right)\,
	\frac{\bd\G}{\bd\z}\left(\th,b\right)
	\cdot
	e^{\i\th}\\
	&&\nonumber\\
	&=&
	\left(
		\tau\cdot\frac{\bd\G}{\bd\th}\left(\th,b\right)
	\right)\,
	\left(
		\frac{\bd\G}{\bd\z}\left(\th,b\right)\cdot e^{\i\th}
	\right)\,.\nonumber
\end{eqnarray}
\vspace{5mm}

\textit{Right side of (\ref{eqn:star}):} 
Now differentiate the right side of (\ref{eqn:star}). Again, the constancy of $\,\z(\eps,\th)\,$ at 
$\,\eps=0\,$ eliminates most summands, so that

\begin{eqnarray}\label{eqn:rhs'}
	\lefteqn{
		\frac{\bd}{\bd\eps}\Big|_{\eps=0}
		\left[-
			\frac{\bd\G}{\bd\th}\Bigl(\theps,\,\z\left(\eps,\theps\right)\Bigr)\,\cdot e^{\i\th}
		\right] =
	}\hspace{0.4in}\nonumber\\
	&&\nonumber\\
	&&
	-\left(
		\frac{\bd^{2}\G}{\bd\th^{2}}\bigl(\th,b\bigr)\cdot e^{\i\th}
	\right)
	\frac{\bd\theps}{\bd\eps}\Big|_{\eps=0}
	\ -\
	\left(
		\frac{\bd^{2}\G}{\bd\th\,\bd\z}\bigl(\th,b\bigr)\cdot e^{\i\th}
	\right)
	\frac{\bd\z}{\bd \eps}\bigl(0,\th\bigr)\,.
\end{eqnarray}
\vspace{2mm}

We can now simplify this further, because $\,\G(\,\cdot, b)\,$ support-parametrizes $\,\ov(b)\,$.
This implies, via (\ref{eqn:spt}),  that at the preimage $\,(\th,b)\,$ of any point in that oval,
we have two identities:

\begin{equation*}
	\frac{\bd\G}{\bd\th}\cdot e^{\i\th}\equiv 0
	\quad\text{and}\quad
	\frac{\bd\G}{\bd\th}\cdot \i\,e^{\i\th}= \mag{\frac{\bd\G}{\bd\th}}\,.
\end{equation*}

The first of these lets us deduce

\begin{equation*}\label{eqn:thz}
	\frac{\bd^{2}\G}{\bd\th\,\bd\z}\cdot e^{\i\th}
	=
	\frac{\bd}{\bd\z}\,\left(\frac{\bd\G}{\bd\th}\cdot e^{\i\th}\right) 
	=
	0\,,
\end{equation*}
\vskip 3pt

which eliminates the final term on the right in (\ref{eqn:rhs'}).

Alternatively, if we differentiate the first of the two identities above with respect to $\,\th\,$, and  
then use the second, we get

\begin{equation}\label{eqn:th''}
	\frac{\bd^{2}\G}{\bd\th^{2}}\cdot e^{\i\th} 
	= 
	-\frac{\bd\G}{\bd\th}\cdot\i\,e^{\i\th}
	= 
	-\mag{\frac{\bd\G}{\bd\th}}\,.
\end{equation}

This lets us rewrite the first term on the right in (\ref{eqn:rhs'}), collapsing the whole
equation to

\begin{equation}\label{eqn:rhs'2}
	\frac{\bd}{\bd\eps}\Big|_{\eps=0}
		\left[-
			\frac{\bd\G}{\bd\th}\Bigl(\theps,\,\z\left(\eps,\theps\right)\Bigr)\,\cdot e^{\i\th}
		\right]
	=
	\mag{\frac{\bd\G}{\bd\th}}\ \frac{\bd\theps}{\bd\eps}\Big|_{\eps=0}\,.
\end{equation}
\bigskip

We now finish by setting (\ref{eqn:lhs'}) equal to (\ref{eqn:rhs'2}). This exhibits the initial 
$\eps$-derivative of equation (\ref{eqn:star}) as

\begin{equation*}
	\left(
		\frac{\bd\G}{\bd\th}\cdot\tau
	\right)\,
	\left(
		\frac{\bd\G}{\bd\z}\cdot e^{\i\th}
	\right)
	\ =\
	\mag{\frac{\bd\G}{\bd\th}}\ \frac{\bd\theps}{\bd\eps}\Big|_{\eps=0}\,.
\end{equation*}
\smallskip

Since this holds at the preimage $\,(\th,b)\,$ of any point in $\,\ov(b,\eps)\,$, and since, by
(\ref{eqn:spt}) again, $\,\bd\G/\bd\th\,$ normalizes to $\,\i\,e^{\i\th}\,$, this proves the
last conclusion of our Proposition.
\end{proof}


\subsection{The symmetry obstruction.} 
We shall write $\,\cx_{\eps}\,$ for  the centrix of $\,\ov(b,\eps)\,$. By Observation 
\ref{obs:centrix},  $\,\ov(b,\eps)\,$ is \emph{central} if and only if $\,\cx_{\eps}\,$ is \emph{constant},
or equivalently,
\begin{equation*}
	\frac{\bd}{\bd\th}\cx_{\eps}\equiv 0\,.
\end{equation*}

Now observe that when $\,\ov(b,\eps)\,$ has central symmetry {for all 
$\,\eps\,$ sufficiently near zero}---as it clearly does when $\,\tube\,$ has \cpo---we will 
also have
\begin{equation}\label{eqn:obstruct}
	\frac{\bd^{2}}{\bd\th\,\bd\eps}\Big|_{\eps=0}\cx_{\eps}\equiv 0\,.
\end{equation}

The initial mixed second partial of $\,\cx_{\eps}\,$ thus forms an \emph{obstruction} to \cpo.   

We want to show that conversely, the vanishing of this obstruction---independently 
of the tilt-direction $\,\tau\,$ and the height $\,b\,$ at which we compute it---has a strong
consequence. Indeed, this vanishing condition ultimately yields the partial differential equations 
of Proposition \ref{prop:pdes}, which in turn imply the Splitting Lemma \ref{lem:split}.

To get there, we first need to rewrite the vanishing condition (\ref{eqn:obstruct}) in terms of the 
support function of the horizontal oval $\,\ov(b)\,$. Toward that goal, we abbreviate
\begin{equation*}
	\theps:=\theps(\th)
	\quad\text{and}\quad
	{\bar\th}_{\eps}:=\theps(\th+\pi)
\end{equation*}
for each $\,\th\in\US^{1}\,$, then combine Observation \ref{obs:centrix} with Proposition \ref{prop:theps}
to get a formula for $\,\cx_{\eps}\,$:
\begin{equation}\label{eqn:cxe}
	\cx_{\eps}(\th)
	= 
	\frac{
		\G\left(\theps,\z(\eps,\theps)\right)
		+
		\G\left({\bar\th}_{\eps},\z(\eps,{\bar\th}_{\eps})\right)
	}{2}\,.
\end{equation}

In order to unpack (\ref{eqn:obstruct}), we must differentiate this formula twice: 
First with respect to $\eps\,$, and then with respect to $\,\th\,$. We record
the initial $\eps$-derivative as Lemma \ref{lem:dceps} below.

To prepare, 
let $\,\G^{*}(\,\cdot, z)\,$ denote the \emph{odd} part of $\,\G(\,\cdot,z)\,$ as specified by
Definition \ref{def:eo}, and let $\,\cx(z)\,$ denote the centroid of $\,\ov(z)\,$ for each 
$\,-1<z<1\,$.  In the language of Definition \ref{def:tct}, $\,\cx\,$ parametrizes the
\emph{central curve} of $\,\tube\,$, while $\,\G^{*}\,$ parametrizes its \emph{rectification}
$\,\tube^{*}\,$.

\begin{lem}\label{lem:dceps}
	Suppose the horizontal cross-section $\,\ov(z)\,$ of a transversely convex tube 
	$\,\tube\,$ is central about $\,(\cx(z),z)\,$ for each $\,-1<z<1\,$. 
	Then for any fixed tilt-direction 	$\,\tau\in\US^{1}\,$, we have

	\begin{eqnarray*}
			\frac{\bd\cx_{\eps}}{\bd\eps}\left(\th,z\right)\Big|_{\eps=0}
		&=&
		\left(\tau\cdot \i\,e^{\i\th}\right)\Bigl(\frac{\ \bd\G^{*}}{\bd\z}\cdot e^{\i\th}\Bigr)
		\frac{\ \bd\G^{*}}{\bd\th}\\
		&&\\
		&&\qquad \ + \  
		\bigl(\tau\cdot\cx(z)\bigr)\,\cx'(z)
		\ +\ 
		\bigl(\tau\cdot\G^{*}\bigr)\,\frac{\ \bd\G^{*}}{\bd\z}\,.
	\end{eqnarray*}

	We evaluate $\,\G^{*}\,$ and its derivatives here at $\,(\th,z)$ throughout.
\end{lem}

\begin{proof}
With (\ref{eqn:cxe}) in view, we first compute the initial $\eps$-derivative of 
$\,\G(\theps,\,\z(\eps,\theps))$. Recall that by Proposition \ref{prop:theps}, 
$\,\th_{0}(\th) = \th\,$, and abbreviate
\[	
	\th_{0}':=\frac{\bd\theps}{\bd\eps}\Big|_{\eps=0}\,.
\]

A routine application of the chain rule then gives

\begin{eqnarray}
	\lefteqn{\frac{\bd}{\bd\eps}\Big|_{\eps=0}\G\bigl(\theps,\,\z\left(\eps,\theps\right)\bigr)}
	\qquad\qquad\nonumber\\
	&=&
	\frac{\bd\G}{\bd\th}\left(\th,z\right)\th_{0}' 
		+ 
		\frac{\bd\G}{\bd\z}(\th,z)\Bigl(
			\frac{\bd{\z}}{\bd\eps}(0,\th) + \frac{\bd\z}{\bd\th}\left(0,\th\right)\th'_{0}
		\Bigr)\nonumber\\ 
	&&\label{eqn:Geps}\\
	&=&
	\frac{\bd\G}{\bd\th}\left(\th,z\right)\th_{0}' 
	+ 
	\frac{\bd\G}{\bd\z}(\th,z)\Bigl(\tau\cdot\G\left(\th,z\right)\Bigr)\,,
	\nonumber
\end{eqnarray}
where we have used equation (\ref{eqn:zeta=z}) 
and Observation \ref{obs:zeta'} to evaluate the derivatives of $\,\z\,$.

We must average (\ref{eqn:Geps}) over $\,\{\th,\bar\th\}\,$ to get the initial 
$\eps$-derivative of $\,\cx_{\eps}\,$ via (\ref{eqn:cxe}). We assume $\,\G(\th,z)\,$ 
support-parametrizes an oval $\,\ov(z)\,$ having central symmetry about $\,\cx(z)\,$ 
for each $\,-1<z<1\,$, so we have
\begin{equation*}
	\G(\th,z) = \cx(z) + \G^{*}(\th,z)
\end{equation*}
as in Observation \ref{obs:centrix}. Here $\,\G^{*}\,$ and all its $\th$-derivatives are \textbf{odd},
so that for instance
\begin{equation*}
	\G^{*}(\bar\th,z) = -\G^{*}(\th,z)\,.
\end{equation*}
All $\th$-derivatives of $\,\cx(z)\,$, on the other hand, clearly vanish. If we average (\ref{eqn:Geps}) over 
$\,\{\th,\bar\th\}\,$ with all these facts in mind, we get\smallskip
\begin{eqnarray*}
		\frac{\bd\cx_{\eps}}{\bd\eps}\left(\th,z\right)\Big|_{\eps=0}
	&=&
	\frac{1}{2}
	\left\{
		\frac{\bd\G}{\bd\th}\left(\th,z\right)\,\th_{0}' 
			+ 
		\frac{\bd\G}{\bd\z}\left(\th,z\right)\left(\tau\cdot\G\left(\th,z\right)\right)\right. \\
	&& \quad + \left.\quad
		\frac{\bd\G}{\bd\th}\left(\bar\th,z\right)\,\bar\th_{0}' 
			+ 
		\frac{\bd\G}{\bd\z}\left(\bar\th,z\right)\left(\tau\cdot\G\left(\bar\th,z\right)\right)
	\right\}\\
	&&\\
	&=&
	\frac{1}{2}
	\left\{
		\frac{\ \bd\G^{*}}{\bd\th}\,\th_{0}' 
			+ 
		\left(\cx'(z) +\frac{\ \bd\G^{*}}{\bd\z}\right)\Bigl(\tau\cdot\left(\cx(z)+\G^{*}\right)\Bigr)\right. \\
	&& \quad - \left.
		\frac{\ \bd\G^{*}}{\bd\th}\,\bar\th_{0}' 
			+ 
		\left(\cx'(z) -\frac{\ \bd\G^{*}}{\bd\z}\right)\Bigl(\tau\cdot\left(\cx(z)-\G^{*}\right)\Bigr)
	\right\},
\end{eqnarray*}

where we now evaluate $\,\G^{*}\,$ and its derivatives at $\,(\th,z)\,$ throughout. 
To simplify further, note that the four mixed products involving $\,\cx$ and 
$\,\G^{*}$-terms cancel in pairs, so that

\begin{eqnarray*}
	\lefteqn{
		\frac{\bd\cx_{\eps}}{\bd\eps}\left(\th,z\right)\Big|_{\eps=0}
	}\qquad\\
	&&\\
	&=&
	\left(\frac{\th_{0}'-\bar\th_{0}'}{2}\right)\frac{\ \bd\G^{*}}{\bd\th}
	+ 
	\bigl(\tau\cdot\cx(z)\bigr)\cx'(z)
	+
	\left(\tau\cdot\G^{*}\right)\frac{\ \bd\G^{*}}{\bd\z}	
\end{eqnarray*}

This will give the formula we seek---we just need to prove
\begin{equation}\label{eqn:diff}
	\frac{\th_{0}'-\bar\th_{0}'}{2} 
	= 
	\left(\tau\cdot \i\,e^{\i\th}\right)\Bigl(\frac{\ \bd\G^{*}}{\bd\z}\cdot e^{\i\th}\Bigr)\,.
\end{equation}
For that, we invoke Proposition \ref{prop:theps}.
Since $\,\G^{*}\,$ and $\,e^{\i\th}\,$ are both odd, that Proposition yields

\begin{eqnarray*}
	\th_{0}' 
	&=& 
	\left(\tau\cdot \i\,e^{\i\th}\right)\Bigl(\cx'(z)\cdot e^{\i\th} + \frac{\ \bd\G^{*}}{\bd\z}\cdot e^{\i\th}\Bigr)\\
	&&\\
	\bar\th_{0}'
	&=&
	\left(\tau\cdot \i\,e^{\i\th}\right)\Bigl(\cx'(z)\cdot e^{\i\th} - \frac{\ \bd\G^{*}}{\bd\z}\cdot e^{\i\th}\Bigr)\,.
\end{eqnarray*}

Subtract the second line from the first to get (\ref{eqn:diff}), and the desired formula follows. 
\end{proof}

To finish analyzing the vanishing condition (\ref{eqn:obstruct}), 
we next need to differentiate the result just proven with respect to $\th$. 
That seems to require a lengthy calculation, but if we work with respect 
to the frame $\,\{e^{\i\th},\,\i\,e^{\i\th}\}\,$, a simple observation eliminates the 
$\,e^{\i\th}$ term entirely.

\goodbreak
\begin{obs}\label{obs:ftau}
	Suppose the horizontal cross-section $\,\ov(z)\,$ of a transversely convex tube 
	$\,\tube\,$ is central about $\,(\cx(z),z)\,$ for each $\,-1<z<1\,$. 
	Then for each tilt-direction $\,\tau\in\US^{1}\,$, there exists a function 
	$\,f_{\tau}:\US^{1}\times(-1,1)\to\R\,$ such that
	\begin{equation*}
		\frac{\,\bd^{2}\cx_{\eps}}{\bd\eps\,\bd\th}\bigl(\th,z\bigr)\Big|_{\eps=0}
		=
		f_{\tau}(\th,z)\,\i\,e^{\i\th}
	\end{equation*}
	for all $\,(\th,z)\in\US^{1}\times(-1,1)\,$.

\end{obs}
\smallskip

\begin{proof}
We get $\,\cx_{\eps}\,$ by symmetrizing each member in a smooth family of support 
parametrizations:
\begin{equation*}
	\cx_{\eps}\,(\th) = \half\left(\g_{\eps}(\th)+\g_{\eps}(\th+\pi)\right)
\end{equation*}
Indeed, our formula (\ref{eqn:cxe}) expresses $\,\cx_{\eps}\,$ in this way. It then follows
from the defining condition (\ref{eqn:spt}) for support parametrizations, that
\begin{equation*}
	\frac{\bd}{\bd\th}\cx_{\eps} = \half\bigl(\mag{\g_{\eps}'(\th)}+\mag{\g_{\eps}'(\th+\pi)}\bigr)\i\,e^{\i\th}
\end{equation*}
Differentiation with respect to $\,\eps\,$ affects only the scalar coefficient of $\,\i\,e^{\i\th}\,$ here, 
making the desired fact obvious.
\end{proof}

Thanks to Observation \ref{obs:ftau}, the vanishing condition (\ref{eqn:obstruct}) reduces
to $\,f_{\tau}\equiv 0\,$.  The two crucial PDE's we have been aiming toward merely interpret 
this simple equation and now make their appearance in the statement of Proposition \ref{prop:pdes} below.

As we have explained above, Proposition \ref{prop:pdes} is the technical heart of this section. 
It also marks our first real use of the \cpo\ assumption: Up to now, our results have at most
assumed central symmetry for the \emph{horizontal} cross-sections of $\,\tube\,$.

To set up the statement of Proposition \ref{prop:pdes}, recall that for each $\,|z|<1\,$,
$\,\G^{*}(\,\cdot, z)\,$ support-parametrizes the horizontal cross-section $\,\ov(z)-\cx(z)\,$
of the rectified tube $\,\tube^{*}\,$. There consequently exists a $\,C^{2}\,$ function
\begin{equation*}
	h:\US^{1}\times[-1,1]\to\R
\end{equation*}
which, for each fixed $\,|z|<1\,$, yields the support function of that oval. We 
call $\,h\,$ the \emph{transverse support function of $\,\tube^{*}\,$}.

To simplify notation, we now adopt the convention of indicating partial differentiation
with respect to a given variable by subscripting with that variable.

\begin{prop}\label{prop:pdes}
	On a transversely convex tube $\,\tube\,$ with \cpo, the transverse support 
	function $\,h\,$ of $\,\tube^{*}\,$ satisfies two partial differential equations:
	\medskip	
	\begin{equation*}
		\bigl(h_{z}\left(h+h_{\th\th}\right)\bigr)_{\th}
			+
		\bigl(h_{\th}\left(h+h_{\th\th}\right)\bigr)_{z}
		=0
	\end{equation*}
	and
	\begin{equation*}
		h\,\left(h+h_{\th\th}\right)_{z}
		- 
		\left(h+h_{\th\th}\right)\,h_{z}
		=0\,.
	\end{equation*}
\end{prop}

\begin{proof}
	Differentiation with respect to $\,\th\,$ annihilates $\,\cx\,$ and $\,\cx'\,$, and
	hence Lemma \ref{lem:dceps} combines with Observation \ref{obs:ftau} 
	to give
	\begin{eqnarray*}
		f_{\tau}
		&=&
		\i\,e^{\i\th}\cdot \frac{\bd^{2}\cx_{\eps}}{\bd\eps\,\bd\th}\Big|_{\eps=0}\\
		&&\\
		&=&
		\i\,e^{\i\th}\cdot
		\bigl[
			\left(\tau\cdot\i\,e^{\i\th}\right)\left(\G^{*}_{z}\cdot e^{\i\th}\right)\G^{*}_{\th}
		\bigr]_{\th}
		+
		\i\,e^{\i\th}\cdot
		\bigl[
			\left(\tau\cdot\G^{*}\right)\G^{*}_{z}
		\bigr]_{\th}
	\end{eqnarray*}
	Since $\,\G^{*}\,$ support-parametrizes $\,\ov(z)-\cx(z)\,$ for each $\,z\,$, however, 
	we have $\,\G^{*}_{\th}=\mag{\G^{*}_{\th}}\i\,e^{\i\th}\,$. This is perpendicular
	to $\,-e^{\i\th}=\left(\i\,e^{\i\th}\right)_{\th}\,$,  so the product rule 
	lets us rewrite the first term on the right above as\medskip
	\begin{equation*}
		\bigl[
			\left(\tau\cdot\i\,e^{\i\th}\right)\left(\G^{*}_{z}\cdot e^{\i\th}\right)
			\mag{\G^{*}_{\th}}
		\bigr]_{\th}\,.
	\end{equation*}

	To evaluate the second term, note that 
	$\,\G^{*}_{\th}\cdot\tau = \mag{\G^{*}_{\th}}\i\,e^{\i\th}\cdot\tau\,$, and\medskip
	\begin{equation*}
		\G^{*}_{z\th}\cdot\i\,e^{\i\th} 
		= 
		\left(\G^{*}_{\th}\cdot\i\,e^{\i\th}\right)_{z}
		=
		\mag{\G^{*}_{\th}}_{z}\,.
	\end{equation*}

	Taking all these facts into account, our expansion of $\,f_{\tau}\,$ becomes\medskip
	\begin{eqnarray*}\label{eqn:216}
		f_{\tau}&=&
		\bigl[
			\left(\tau\cdot\i\,e^{\i\th}\right)\left(\G^{*}_{z}\cdot e^{\i\th}\right)
			\mag{\G^{*}_{\th}}
		\bigr]_{\th}\\
		&&\qquad	+\ 
		\mag{\G^{*}_{\th}}\left(\tau\cdot\i\,e^{\i\th}\right)\G^{*}_{z}\cdot\i\,e^{\i\,\th}
		\ +\ 
		\left(\tau\cdot\G^{*}\right)\mag{\G^{*}_{\th}}_{z}\nonumber
	\end{eqnarray*}
	Now separate multiples of $\,\tau\cdot\i\,e^{\i\th}\,$ from those
	of $\,\tau\cdot e^{\i\th}\,$, noting that
	$\,\left(\tau\cdot\i\,e^{\i\th}\right)_{\th}=-\left(\tau\cdot e^{\i\th}\right)\,$, and
	that by orthonormal expansion,\medskip	
	\begin{equation*}
		\tau\cdot\G^{*}
		= 
		\left(\tau\cdot e^{\i\th}\right)\left(e^{\i\th}\cdot\G^{*}\right)
		+
		\left(\tau\cdot \i\,e^{\i\th}\right)\left(\i\,e^{\i\th}\cdot\G^{*}\right)\,.
	\end{equation*}

	Use these facts to expand $\,f_{\tau}\,$ further, collecting multiples of
	$\,\tau\cdot\i\,e^{\i\th}\,$ and $\,\tau\cdot e^{\i\th}\,$, and noticing that\medskip	
	\begin{equation*}
		\left(\G^{*}\cdot\i\,e^{\i\th}\right)\mag{\G^{*}_{\th}}_{z}
		+
		\left(\G^{*}_{z}\cdot \i\,e^{\i\th}\right)\mag{\G^{*}_{\th}}
		=
		\bigl(\left(\G^{*}\cdot\i\,e^{\i\th}\right)\mag{\G^{*}_{\th}}\bigr)_{z}
	\end{equation*}
	to get
	\begin{eqnarray}\label{eqn:ftau}
		\lefteqn{f_{\tau}	=}\nonumber\\
		&&
		\left(\tau\cdot\i\,e^{\i\th}\right)
		\Bigl[
			\bigl(\left(\G^{*}_{z}\cdot e^{\i\th}\right)\mag{\G^{*}_{\th}}\bigr)_{\th}
			+
			\bigl(\left(\G^{*}\cdot\i\,e^{\i\th}\right)\mag{\G^{*}_{\th}}\bigr)_{z}
		\Bigr]
		\\
		&&\qquad\qquad +
		\left(\tau\cdot e^{\i\th}\right)
		\Bigl[
			\left(e^{\i\th}\cdot\G^{*}\right)\mag{\G^{*}_{\th}}_{z}
			-
			\left(\G^{*}_{z}\cdot e^{\i\th}\right)\mag{\G^{*}_{\th}}
		\Bigr]\,.\nonumber
	\end{eqnarray}

	Now we invoke the central plane oval assumption, observing that
	\emph{when $\,\tube\,$ has \cpo, we must have $\,f_{\tau}\equiv 0\,$}.
	
	Indeed, \cpo\ endows the tilted ovals $\,\bar\ov(z,\eps)\,$ with
	central symmetry for all $\,\tau\in\US^{1}\,$, all $\,-1<z<1\,$, and all 
	sufficiently small $\,\eps\,$. As noted earlier, the projected ovals 
	$\,\ov(z,\eps)\,$ inherit that symmetry too, since the projection
	$\,(x,y,z)\to(x,y)\,$ induces an affine isomorphism from any non-vertical plane
	to $\,\R^{2}\,$. 
	
	Observation \ref{obs:centrix} then makes the centrix $\,\cx_{\eps}\,$
	of $\,\ov(z,\eps)\,$ constant (i.e. independent of $\,\th\,$) for any tilt-direction
	$\,\tau\,$, any $\,|z|<1\,$ and all any sufficiently small $\,\eps\,$.
	The vanishing condition (\ref{eqn:obstruct}) therefore obtains. Given
	Observation \ref{obs:ftau}, this forces $\,f_{\tau}\equiv 0\,$ as claimed.
	
	We may consequently set the right-hand side of (\ref{eqn:ftau}) equal to
	zero. But the resulting identity holds for \emph{any} tilt-direction $\,\tau=:e^{\i\phi}\in\US^{1}\,$,
	and \emph{the coefficients $\,\tau\cdot e^{\i\th}=\cos(\phi-\th)\,$ and 
	$\,\tau\cdot \i\,e^{\i\th}=\sin(\phi-\th)\,$ appearing there are clearly linearly independent 
	functions of $\,\tau\,$.} The terms they multiply must therefore vanish \emph{individually}. 
	In short, we now have
	\medskip
	\begin{eqnarray}
		0&=&
			\Bigl(\left(\G^{*}_{z}\cdot e^{\i\th}\right)\mag{\G^{*}_{\th}}\Bigr)_{\th}
			+
			\Bigl(\left(\G^{*}\cdot\i\,e^{\i\th}\right)\mag{\G^{*}_{\th}}\Bigr)_{z}
		\label{eqn:a}
		\\
		0&=&
			\left(e^{\i\th}\cdot\G^{*}\right)\mag{\G^{*}_{\th}}_{z}
			-
			\left(\G^{*}_{z}\cdot e^{\i\th}\right)\mag{\G^{*}_{\th}}
		\label{eqn:b}
	\end{eqnarray}

	For each $\,|z|<1\,$, the relationship between the support parametrization $\,\G^{*}(\,\cdot,z)\,$ 	
	of $\,\ov(z)\,$ and its support function $\,h(\,\cdot,z)\,$, as detailed in \S\ref{sec:ovals}, 
	now lets us write
	\begin{equation*}
		\G^{*} = \left(h+\i\,h_{\th}\right)\,e^{\i\th}
		\quad\text{and}\quad
		\G^{*}_{\th} = \left(h+h_{\th\th}\right)\i\,e^{\i\th}\,,
	\end{equation*}
	from which we can immediately deduce
	\begin{equation*}
		\begin{array}{rclcccl}
			\G^{*}\cdot e^{\i\th}&=& h\,,
			&\quad&
			\G^{*}\cdot\i\,e^{\i\th}&=&h_{\th}\\
			\G^{*}_{z}\cdot e^{\i\th} &=& h_{z}\,,
			&\quad&
			\mag{\G^{*}_{\th}}&=&h+h_{\th\th}\,,
		\end{array}
	\end{equation*}

	Substituting these into  (\ref{eqn:a}) and (\ref{eqn:b}) instantly gives the 
	differential equations we want.
\end{proof}

We can now prove our Splitting Lemma \ref{lem:split}, restated below. As above, 
$\,h\,$ denotes the transverse support function of $\,\tube^{*}\,$, the rectification
of a transversely convex tube $\,\tube\,$ with central curve $\,\cx\,$. Recall that we
say $\,\tube^{*}\,$ \emph{splits} if we can factor its support map $\,\G^{*}(z,\th)\,$ as a
product $\,\g(\th)r(z)\,$, with $\,\g\,$ parametrizing a fixed oval and $\,r>0\,$.\medskip

\begin{prop}[Splitting Lemma]\label{lem:split}
	If a transversely convex tube $\,\tube\,$ in standard position has \cpo,
	then its rectification $\,\tube^{*}\,$ splits.
\end{prop}

\begin{proof}
	It will clearly suffice to prove that the transverse support function $\,h\,$
	of $\,\G^{*}\,$ factors as $\,h(z,\th)=h(\th)\,r(z)\,$. We know that
	$\,h(z,\th)\,$ satisfies the two differential equations of Proposition
	\ref{prop:pdes}, and we start by noticing that the second equation there 
	forms the numerator of a quotient-rule calculation. 
	Specifically, it implies\medskip
	\begin{equation*}
		\frac{\bd}{\bd z}
		\left(
			\frac{h+h_{\th\th}}{h}
		\right)
		= 0\,,
	\end{equation*}
	from which we easily deduce
	\begin{equation}\label{eqn:floq}
		h_{\th\th}+h = q^{2}(\th)\,h
	\end{equation}
	for some \emph{strictly} positive, $z$-independent function $\,q\,$ on $\,\US^{1}\,$.  
	We can assume positivity of $\,q\,$ because
	$\,\ov(z)-\cx(z)\,$ is origin-centered and strictly convex for each $\,z\,$, properties 
	that, by equations (\ref{eqn:sptfn1}) and (\ref{eqn:speed}), make both $\,h\,$ and 
	$\,h_{\th\th}+h\,$ strictly positive.
	
	In any case, since $\,q\,$ depends only on $\,\th\,$, we see that the support functions 
	of the translated ovals $\,\ov(z)-\cx(z)\,$ all solve the same ordinary differential equation, 	namely (\ref{eqn:floq}). Such equations have independent solutions, of course, so by 
	itself, (\ref{eqn:floq}) leaves us short of splitting. But it lets us rewrite the 
	\emph{first} differential equation of Proposition \ref{prop:pdes} as
	\begin{equation}\label{eqn:ode2'}
		\bigl(h_{z}\,h\,q^{2}\bigr)_{\th} + \bigl(h_{\th}\,h\,q^{2}\bigr)_{z}=0\,.
	\end{equation}

	Since $\,h_{z}h\,$ and $\,h_{\th}h\,$ are derivatives of (half) the
	\emph{squared} support function 
	\begin{equation*}
		H(\th,z):=h^{2}(\th,z)\,,
	\end{equation*}
	we can the exploit $z$-independence of $\,q\,$, and use $\,H_{z\th}=H_{\th z}\,$ 
	to rewrite (\ref{eqn:ode2'}) in the form of a first-order equation for $\,H_{z}\,$:
	\begin{equation*}
		2H_{z\th}\,q^{2}+H_{z}\bigl(q^{2}\bigr)_{\th} = 0\,.
	\end{equation*}
	Now multiply by $\,H_{z}\,$ to recognize that (\ref{eqn:ode2'}) actually reduces 
	to
	\begin{equation*}
		\bigl(H_{z}^{2}q^{2}\bigr)_{\th}=0\,.
	\end{equation*}
	Evidently, there exists a $\th$-independent function $\,\phi(z)\,$ such that
	\begin{equation*}
		H_{z}(\th,z) = \phi(z)\big/q(\th)\,.
	\end{equation*}

	Integrating with respect to $\,z\,$ then yields\medskip
	\begin{equation*}
		H(\th,z) = H(\th,0) + \frac{\Phi(z)}{q(\th)}\,,
		\quad\text{where}\quad
		\Phi(z):=\int_{0}^{z}\phi(s)\,ds\,.
	\end{equation*}
	Rewrite this as  
	\begin{equation*}
		H(\th,z) = H(\th,0)\left(1 + \a(\th)\Phi(z)\right)\,,
	\end{equation*}
	where
	\begin{equation*}
		\a(\th):=\frac{1}{H(\th,0)\,q(\th)}\,.
	\end{equation*}
	Since $\,H=h^{2}\,$, and, as the support function of an origin-centered oval, $\,h(\th,z)$
	is always positive, we see that $\,1+\a\,\Phi>0\,$ too. Hence
	\begin{equation}\label{eqn:presplit}
		h(\th,z) = h(\th,0)\sqrt{1+\a(\th)\,\Phi(z)}\,.
	\end{equation}

	The continuity of $\,\a\,$ guarantees it a maximum value $\,\bar\a\,$ 
	at some point $\,\bar\th\in\US^{1}\,$, and there, (\ref{eqn:presplit}) yields
	\medskip
	\begin{eqnarray*}
		h(\bar\th,z) 
		&=& h(\bar\th,0)\sqrt{1+\bar\a\,\Phi(z)}\\
		h_{\th}(\bar\th,z) 
		&=& h_{\th}(\bar\th,0)\sqrt{1+\bar\a\,\Phi(z)}\,.\rule{0pt}{6mm}
	\end{eqnarray*}

	These identities show that for any fixed $\,z\,$ with $\,|z|<1\,$, 
	the functions $\,h(\th,z)\,$ and 
	$\,h(\th,0)\sqrt{1+\bar\a\,\Phi(z)}\,$ both obey the same initial conditions at $\,\th=\bar\th\,$. 
	Since both also solve (\ref{eqn:floq}), Picard's uniqueness theorem forces them to agree 
	everywhere. 
	
	The Lemma consequently holds with
	\medskip
	\begin{equation*}
		r(z) = \sqrt{1+\bar\a\,\Phi(z)}
		\quad\text{and}\quad
		h(\th) = h(\th,0)\,.
	\end{equation*}
\end{proof}

We now reach the main goal of this section---a geometric consequence of the Splitting lemma:

\begin{prop}\label{prop:branch}
	Suppose $\,\tube\,$ is a transversely convex tube with \cpo.
	Then its rectification $\,\tube^{*}\,$ is either
	\begin{itemize}
		\item[(i)] The cylinder over a central oval, or
		\smallskip
		\item[(ii)] Affinely congruent to a surface of revolution.
	\end{itemize}
\end{prop}

\begin{proof}
We show that when $\,\tube\,$ is a transversely convex tube in standard position,
and $\,\tube^{*}\,$ is \emph{not} a cylinder, there exists a single linear isomorphism that
fixes the $z$-axis while making each horizontal cross-section $\,\ov(z)\,$ of $\,\tube^{*}\,$ simultaneously circular. This clearly implies the desired result.

We start by using the Splitting Lemma to factor the transverse support function 
$\,h\,$ of $\,\tube^{*}\,$ as
\begin{equation}\label{eqn:fac}
	h(\th,z) = r(z)\,h(\th)\,.
\end{equation}
Put this factorization \emph{back} into the first differential equation in 
Proposition \ref{prop:pdes} and simplify, to find that $\,r\,$ and $\,h\,$ now jointly solve
\begin{equation}\label{eqn:theguy}
	r\,r'\bigl(h\,h''' +3\,h'h'' + 4\,h\,h'\bigr) = 0
\end{equation}
on $\,\US^{1}\times (-1,1)\,$. We have assumed that $\,\tube^{*}\,$ is not cylindrical, 
so $\,r'(z_{0})\ne 0\,$ for some $\,-1<z_{0}<1\,$. Evaluating (\ref{eqn:theguy}) 
at that height, we then deduce that the horizontal support function $\,h(\th)\,$ solves 
the following ordinary differential equation:
\begin{equation*}
	h\,h''' +3\,h'h'' + 4\,h\,h'=0\,.
\end{equation*}
The reader will find it routine to verify what came as a pleasant surprise to us: 
That this quadratic ODE for $\,h\,$ reduces to a linear equation---one that could 
hardly be more familiar---for the \emph{squared} support function $\,H(\th):=h^{2}(\th)\,$:
\begin{equation*}
	H''' + 4H' = 0\,.
\end{equation*}
By Proposition \ref{prop:H'''}, this makes $\,h(\th)\,$ the support function of an 
origin-centered \emph{ellipse} $\,\ov_{0}\,$. By (\ref{eqn:fac}), \emph{every} 
horizontal cross-section of $\,\tube^{*}\,$ is then homothetic to $\,\ov_{0}\,$, and 
since it is origin-centered,  $\,\ov_{0}\,$ is congruent to the unit circle via some 
linear mapping $\,A\,$ of $\,\R^{2}\,$. Extending $\,A\,$ trivially to $\,\R^{3}\,$, we 
clearly map $\,\tube^{*}\,$ to a surface of revolution, precisely as we sought to prove.	
\end{proof}


\goodbreak

\section{Straightening the central curve}\label{sec:axis}

So far we have shown, using variational and analytic arguments, that when a transversely convex 
tube has \cpo, it rectifies to either a cylinder or---up to affine isomorphism---a surface of revolution.
We now use more elementary arguments of a local geometric type to show that the rectification step 
is actually superfluous. Specifically, we prove\medskip

\begin{prop}[Axis lemma]\label{prop:axis}
	Suppose $\,\tube\,$ is a transversally convex tube with \cpo. 
	Then its central curve is affine, so that $\,\tube\,$ is affinely congruent
	to its rectification $\,\tube^{*}\,$.
\end{prop} 

\begin{proof}
We can assume $\,\tube\,$ lies in the standard position described by Definition \ref{def:tct}, 
and it clearly suffices to prove that when $\,\tube^{*}\,$ is either a cylinder or a surface
of revolution, \cpo\ forces the axis of $\,\tube\,$ itself to be a straight line. The latter occurs 
if and only if the tube's central curve $\,\cx:(-1,1)\to\R^{2}\,$ is affine (linear plus
constant). We will establish exactly that, using the following

\textbf{Linearity Criterion:}
	\emph{A $\,C^{2}\,$ mapping $\,\cx:I\to\R\,$ is affine on an open interval $\,I\,$
	if and only if it is locally \emph{odd} around each input, in the sense that for all $\,b\in I\,$, 
	we have
	\begin{equation}\label{eqn:lc}
		\cx(b+t)-\cx(b) = -\bigl(\cx(b-t)-\cx(b)\bigr)
	\end{equation}
	for all sufficiently small $\,t\,$.}

When $\,\cx\,$ is affine, (\ref{eqn:lc}) clearly holds. To prove the converse, it suffices to show
that (\ref{eqn:lc}) implies $\,\cx''\equiv 0\,$. But that follows instantly if we differentiate it twice, and 
then let $\,t\to 0\,$.
\medskip

With this criterion in hand, we proceed, treating the cylindrical and rotationally 
symmetric cases separately.

\textbf{Cylindrical case:} 
When $\,\tube^{*}\,$ is a cylinder, its horizontal cross-section $\,\ov(z)$ at every height
$\,z\in (-1,1)\,$ translates to a fixed central oval $\,\ov_{0}\in\R^{2}\,$. Take $\,\ov_{0}\,$
to be centered at the origin and denote its support parametrization by $\,\g\,$ to get this
parametrization $\,X:\US^{1}\times(-1,1)\to\tube\,$:
\begin{equation}\label{eqn:cyl}
	X(t,z) = \left(\cx(z)+\g(t),\,z\right)\,.
\end{equation}

Now consider, for any height $\,b\in(-1,1)\,$, and any angle $\,\th\in\R\,$, the $\th$-diameter of 
$\,\ov(b)$ (Definition \ref{def:centrix}). Since $\,\g\,$ support-parametrizes $\,\ov_{0}\,$, 
the endpoints of this diameter clearly lie at $\,X(\th,b)\,$ and $\,X(\th+\pi, b)\,$, and the crucial point 
is that \emph{the tangent planes to $\,\tube\,$ at these endpoints are parallel.} 
To see that, compute the partial derivatives $\,X_{t}\,$ and $\,X_{z}\,$ at these points. Since
$\,\ov_{0}\,$ is central, we have $\,\g'(\th+\pi)=-\g'(\th)\,$, and this makes the tangent planes
parallel, since both are spanned by
\begin{equation*}
	\left(\g'(\th),\,0\right) = \pm X_{t}
	\quad\text{and}\quad
	\bigl(\cx'(b),\,1\bigr) = X_{z}\,.
\end{equation*}

Now suppose, fixing the $\th$-diameter of $\,\ov(b)$ as axis, we tilt the plane $\,z=b\,$ 
away from the horizontal with some small slope $\,\eps>0\,$ to get a new plane 
$\,P_{\eps}(\th)\,$. For sufficiently small $\,\eps>0\,$, the intersection 
$\,\ov(b,\th,\eps):=\tube\cap P_{\eps}(\th)\,$ will remain an oval---and a central oval, since 
$\,\tube\,$ has \cpo. 

Further, since $\,P_{\eps}(\th)\,$ contains the $\th$-diameter of $\,\ov(b)$, the endpoints 
$\,X(\th,b)\,$ and $\,X(\th+\pi,b)\,$ of that diameter remain on $\,\ov(b,\th,\eps)\,$ 
independently of $\,\eps\,$.  And since the tangent planes to $\,\tube\,$ at 
these points are parallel, and their intersections with $\,P_{\eps}(\th)\,$ clearly form lines tangent 
to $\,\ov(b,\th,\eps)\,$ at $\,X(\th,b)\,$ and $\,X(\th+\pi,b)\,$, \emph{those tangent lines are parallel}.

The latter fact shows that the $\th$-diameter of $\,\ov(b)\,$ \emph{remains} a diameter 
of $\,\ov(b,\th,\eps)\,$ independently of $\,\eps\,$, and hence that \emph{$\,(\cx(b),b) $ forms 
the center of $\,\ov(b,\th,\eps)\,$}, for each $\,\th\,$ and each sufficiently small $\,\eps>0\,$. 
The center of $\,\ov(b,\th,\eps)\,$ remains fixed as we vary $\,\eps\,$.
 
Now observe that every point sufficiently close to $\,\ov(b)$ on $\,\tube\,$ belongs 
$\,\ov(b,\th,\eps)\,$ for some $\,\th\,$ and some small $\,\eps>0\,$, so that by \cpo, its
reflection through $\,(\cx(b),b)\,$ also lies on $\,\tube\,$. It follows that an entire neighborhood 
of $\,\ov(b)$ in $\,\tube\,$ has reflection symmetry through $\,(\cx(b),b)\,$. 
In some neighborhood of $\,(\cx(b),b)\,$, the central curve $\,\cx\,$ of $\,\tube\,$ then inherits that
same reflection symmetry. Since $\,b\in(-1,1)\,$ was arbitrary, this clearly means that (\ref{eqn:lc}) 
holds for $\,\cx\,$, and our Linearity Criterion now straightens the central curve, as desired.
\medskip

\textbf{Surface-of-revolution case:}
Here, each horizontal plane $\,z\equiv b\,$ cuts  the original tube $\,\tube\,$
 in a \emph{circle} centered at $\,(\cx(b),b)\,$ for each $\,b\in(-1,1)\,$. 
Write $\,F(b)>0\,$ for the squared radius of this circle, and $\,(\xi(b),\eta(b)):=\cx(b)\,$
for the horizontal coordinates of its center. Then $\,\tube\,$ clearly constitutes the solution 
set of\smallskip
\begin{equation}\label{eqn:T}
	\left(x-\xi(z)\right)^{2}+\left(y-\eta(z)\right)^{2}=F(z)\,.
\end{equation}
The  $\,C^{2}\,$ differentiability of $\,\tube\,$ ensures that $\,F\,$, $\,\xi\,$ and $\,\eta\,$
are all $\,C^{2}\,$ on $\,(-1,1)\,$. 

We want to show that \cpo\ forces $\,\cx\,$ to be affine. 
To do so, we study the even and odd components of $\,\xi,\,\eta,$ and $\,F\,$ 
with respect to reflection through a point, and for that we introduce the following notation.

Suppose $\,\b\in\R\,$, and let $\,f\,$ denote any function defined on a neighborhood 
of $\,\b\,$.  We define the $\,\b$-\textbf{translate} of $\,f\,$ by
\[
	f_{\b}(t):=f(\b+t)\ .
\]
We also define the \textbf{even} and \textbf{odd} parts of $\,f_{\b}\,$ respectively as
\begin{equation*}
	f_{\b}^{+}(t) = \frac{f_{\b}(t)+f_{\b}(-t)}{2}\ ,\qquad
	f_{\b}^{-}(t) = \frac{f_{\b}(t)-f_{\b}(-t)}{2}\,.
\end{equation*}
As usual, we then have
\[
	f_{\b}^{+}(-t) = f_{\b}^{+}(t)\ ,
	\quad
	f_{\b}^{-}(-t) = -f_{\b}^{-}(t)
\]
and
\[
	f_{\b}(t) = f_{\b}^{+}(t)+f_{\b}^{-}(t)\ ,
	\quad
	f_{\b}(-t) = f_{\b}^{+}(t)-f_{\b}^{-}(t)\ .
\]
%


Now fix an arbitrary height $\,\b\in(-1,1)\,$. Since $\,\tube\,$ is horizontally circular, 
has \cpo, and lies in standard position, we can find a small slope $\,m>0\,$, 
and a $z$-intercept $\,b=b(\b)\,$ such that the plane $\,P\,$ given by\smallskip
\[
	z=mx+b\quad\text{or}\quad x=\frac{z-b}{m}\,,
\]

cuts $\,\tube\,$ is a central oval $\,\ov\,$, depending on $\,m\,$ and $\,\b\,$, and centered 
at height $\,\b\,$. In the $\,(y,z)\,$ coordinate system on $\,P\,$, we get the following equation 
for $\,\ov\,$ by restricting (\ref{eqn:T}):\smallskip
\begin{equation}\label{eqn:pxcaps}
	\left(\frac{z-b}{m}-\xi(z)\right)^{2}+ \left(y-\eta(z)\right)^{2}=F(z)\,.
\end{equation}

Solve this for $\,y\,$ in terms of the $\b$-centered variable $\,t:=z-\b\,$
to split $\,\ov\,$ into a pair of arcs, graphs of functions we shall call $\,y_{\pm}(t)\,$,
over the symmetric interval
\begin{equation}\label{eqn:Tmall}
	|t|<\sup\{z-\b\colon (x,y,z)\in\ov\}\,.
\end{equation}
Using the notation defined above, we can express these functions as\smallskip
\begin{equation}\label{eqn:Tpm}
	y_{\pm}(t) 
	:= 
	\eta_{\b}(t) 
		\pm 
		\sqrt{F_{\b}(t)-\left(\frac{\bar\b+t}{m}-\xi_{\b}(t)\right)^{2}}\,,
\end{equation}

where $\,\bar\b:=\b-b\,$. 

Since the chord joining $\,(y_{+}(0),\b)\,$ to $\,(y_{-}(0),\b)\,$ has height $\,\b\,$, 
it clearly passes through the center of $\,\ov\,$. It must therefore be a diameter. 
But the midpoint of any diameter locates the center of $\,\ov\,$, so using 
(\ref{eqn:Tpm}) to average $\,y_{\pm}(0)\,$, we can now deduce that:

\emph{The center of $\,\ov\,$ has coordinates
$\,(\eta(\b),\b)\,$ in the $\,(y,z)\,$ coordinate system on
$\,P\,$.}

This fact lets us express the central symmetry of $\,\ov\,$ as the coordinate swap
\begin{equation*}
	(\eta(\b)+s,\,\b+t)\ \longleftrightarrow\ (\eta(\b)-s,\,\b-t)\,.
\end{equation*}
When $\,t\,$ is small enough as measured by (\ref{eqn:Tmall}), 
this swap always exchanges diametrically opposed solutions of 
(\ref{eqn:pxcaps}). Write the resulting two statements in terms of the 
notation introduced above to get two simultaneous identities:\smallskip
\begin{eqnarray}\label{eqn:ypt}
	\lefteqn{F^{+}_{\b}(t)+F^{-}_{\b}(t)}\nonumber\\
	&&\quad =\
		\left(\eta(\b)+s-\eta^{+}_{\b}(t)-\eta^{-}_{\b}(t)\right)^{2}\\
	&&\qquad\qquad
		+
		\left(
			\ds{\frac{\bar\b+t}{m}}-\xi_{\b}^{+}(t)-\xi_{\b}^{-}(t)
		\right)^{2}\nonumber
\end{eqnarray}
and
\begin{eqnarray}\label{eqn:ymt}
	\lefteqn{F^{+}_{\b}(t)-F^{-}_{\b}(t)}\nonumber\\
	&&\quad =\
		\left(\eta(\b)-s-\eta^{+}_{\b}(t)+\eta^{-}_{\b}(t)\right)^{2}\\
	&&\qquad\qquad
		+
		\left(
			\ds{\frac{\bar\b-t}{m}}-\xi_{\b}^{+}(t)+\xi_{\b}^{-}(t)
		\right)^{2}\nonumber
\end{eqnarray}

Subtract (\ref{eqn:ymt}) from (\ref{eqn:ypt}), factor differences between
corresponding squares on the right, and divide by two, to obtain
\begin{eqnarray}\label{eqn:Fbm}
	\lefteqn{F^{-}_{\b}(t)}\nonumber\\
		&&\quad =\ 
		2
		\left(\eta(\b)-\eta_{\b}^{+}(t)\right)
		\left(s - \eta_{\b}^{-}(t)\right)\\
		&&\qquad\qquad 
		+\ 2
		\left(\frac{\bar\b}{m}-\xi_{\b}^{+}(t)\right)
		\left(\frac{t}{m}-\xi_{\b}^{-}(t)\right)\nonumber
\end{eqnarray}

The strict convexity of $\,\ov\,$ now guarantees that
the line $\,z=\b+t\,$ in $\,P\,$ cuts $\,\ov\,$ in two distinct points whenever  
$\,t\,$ is sufficiently small. Call the $y$-coordinates of these points 
$\,\eta(\b)+s\,$ and $\,\eta(\b)+s'\,$ respectively. Equation (\ref{eqn:Fbm})
clearly remains true if we replace $\,s\,$ by $\,s'\,$. When we subtract the resulting
$s'$-version of (\ref{eqn:Fbm}) from the $s$-version and simplify, however, we find that for all
sufficiently small $\,t\,$, we have
\begin{equation*}\label{eqn:dFbm}
	\left(s-s'\right)
	\left(\eta_{\b}^{+}(t)-\eta(\b)\right)
	= 
	0\,.
\end{equation*}
Since $\,s\,$ and $\,s'\,$ are distinct for the small $\,t\,$ in question, we evidently must  
have $\,\eta_{\b}^{+}(t)\equiv \eta(\b)\,$ for all sufficiently small $\,t\,$. By definition of
$\,\eta_{\b}^{+}\,$, this means
\[	
	\eta(\b+t)-\eta(\b) = -\bigl(\eta(\b-t)-\eta(\b)\bigr)\,,
\]
so that (\ref{eqn:lc}) holds for $\,\eta\,$. But by swapping the roles of $\,x\,$ and $\,y\,$ in the 
argument above, we find that in precisely the same way, it holds for $\,\xi\,$, and hence for 
$\,\cx = (\xi,\eta)\,$. Our Linearity Criterion then makes the $\,\cx\,$ affine, as desired.
\end{proof}

\goodbreak

\section{Main theorem}\label{sec:main}

By combining the Axis Lemma just proven with our Cylinder/Quadric Proposition \ref{prop:branch} 
and the rotationally invariant case (Proposition \ref{prop:sor}), one immediately deduces

\begin{prop}[Collar Theorem]\label{prop:collar}
	A transversely convex tube with \cpo\ is either cylindrical or quadric.
\end{prop}

We can strengthen this statement substantially, however, without much extra effort:

\begin{thm}[Main Theorem]\label{thm:main}
	A complete, connected $\,C^{2}$-immersed surface in $\,\R^{3}\,$ with  \cpo\ is 
	either a cylinder, or quadric. 
\end{thm}

\begin{proof}
	Suppose $\,F\,$ immerses a complete $\,C^{2}\,$ surface $\,M^{2}\,$ into $\,\R^{3}\,$ 
	with \cpo. The latter assumption ensures, first of all, that $\,F(M)\,$ crosses some 
	affine plane---we take it to be the $\,z=0\,$ plane---transversally (if not exclusively) along 
	a central oval $\,\ov\,$. 
	
	This being the case, define, for any two heights $\,a<0<b\,$, the open connected 	
	component 
	\begin{equation*}
		M_{a,b}\subset F^{-1}\left(\left\{(x,y,z)\in \R^{3}\colon a<z<b\right\}\right)
	\end{equation*}

	as the {unique component containing $\,F^{-1}(\ov)\,$}.
	
	Since $\,\ov\,$ is strictly convex and $\,F(M)\,$ is transverse to the plane $\,z=0\,$ 
	along $\,\ov\,$, standard arguments from basic differential topology show that for 
	$\,a<0<b\,$ sufficiently near $0\,$,
	
	\begin{itemize}
	
		\item[(i)]
		The pullback $\,F^{*}z\,$  of the height function $\,z\,$ on $\,\R^{3}\,$ 
		has no critical points in $\,M_{a,b}\,$, and
		\medskip
		
		\item[(ii)]
		$F\,$ embeds $\,M_{a,b}\,$ in $\,\R^{3}\,$ as a transversely convex tube. 
	\end{itemize}
		
	There consequently exist \emph{minimal} and \emph{maximal} heights 
	$\,-\infty\le A<0<B\le\infty\,$ such that (i) and (ii) above 
	both hold for every finite $\,a<b\,$ in the closed interval $\,[A,B]\,$.
	
	Our proof now forks in three directions, depending on whether both, neither, or
	exactly one of the endpoints $\,A\,$ and $\,B\,$ are finite.
	
	\textbf{Case $\,-\infty<A<B<\infty\,$ (Ellipsoid).} 
	In this case, by (ii), the image of $\,M_{a,b}\,$ under $\,F\,$ is 
	a transversely convex tube for every $\,a<b\,$ in the interval $\,(A,B)\,$.
	This trivially extends to $\,M_{A,B}\,$, and the resulting maximal
	tube clearly inherits \cpo\ from $\,F(M)\,$. Our Collar Theorem \ref{prop:collar} 
	then says that $\,F(M_{A,B})\,$ is either the cylinder on a central oval,
	or quadric. 
	
	We can rule out the first possibility, because on a cylinder, horizontal cross-sections are 
	uniformly convex, and the gradient of $\,z\,$ is bounded away from zero. But these facts, 
	by continuity, would extend slightly beyond $\,A\,$ and $\,B\,$, contradicting their maximality 
	with respect to (i) and (ii) above.
	
	It follows that when $\,-\infty<A<B<\infty\,$, $\,F(M_{A,B})\,$ is quadric. By affine invariance,
	however, we lose no generality by assuming that $\,F\,$ immerses $\,M_{A,B}\,$ as a quadric 
	surface of revolution around the $z$-axis: a vertical segment of an ellipsoid, cone, 
	elliptic paraboloid, or a hyperboloid. On all these surfaces, horizontal cross-sections in any 
	compact slab are uniformly convex. So the maximality of $\,A\,$ and $\,B\,$ must be dictated by 
	condition (i) above, not (ii). The completeness of $\,M\,$, then ensures that  $\,F^{*}z\,$ 
	must have critical points on both boundaries of $\,M_{A,B}\,$. But among the quadrics listed 
	above, $\,z\,$ has multiple critical points only on the ellipsoid, where it attains both a max and 
	a min. The closure of $\,F(M_{A,B})\,$ must therefore be a complete ellipsoid, which, by continuity 
	of $\,F\,$ and connectedness of $\,M\,$ must coincide with $F(M)\,$.
		
	\textbf{Case $\,-A=B = \infty\,$ (Tube hyperboloid or cylinder).}
	In this case we can immediately from the connectedness of $\,M\,$	that $\,M_{A,B}=M\,$. 
	Moreover, since (ii) holds for every finite $\,a<0<b\,$,  $\,F\,$ must embed $\,M_{-r,r}\,$ 
	in $\,\R^{3}\,$ as a transversely convex tube $\,\tube_{r}\,$ for every $\,r>0\,$. 
	As above, $\,\tube_{r}\,$ inherits \cpo\ from $\,F(M)\,$, so by the Collar Theorem \ref{prop:collar}, 
	$\,F\,$ maps $\,M_{-r,r}\,$ to a cylinder over some central oval, or to a non-degenerate quadric, 
	for each $\,r>0\,$. Let $\,S\,$ denote the unique complete unbounded cylinder or quadric that 
	extends $\,F(M_{-1,1})\,$. We then clearly have $\,F(M_{-r,r})=S\,$ in the slab $\,|z|<r\,$ for all 
	$\,r>1\,$. But then $\,S=F(M)\,$ in its entirety, for otherwise, $\,F(M)\,$ deviates from $\,S\,$ 
	at some finite height $\,\rho\,$, a contradition when $\,r>|\rho|\,$. The only smooth quadric
	that contains a horizontal oval and extends infinitely far both above and below the plane $\,z=0\,$
	is the tube hyperboloid. So in this case, $\,M\,$ is either a tube hyperboloid or a cylinder.
		
	\textbf{Cases $\,|A|<B=\infty\,$ or $\,|B|<|A|=\infty\,$ (Paraboloid or convex hyperboloid).}
	Since the reflection $\,z\to-z\,$ is affine, these two cases are equivalent. So we assume 
	$\,|A|<B=\infty\,$, and arguing as in the previous two cases, we now quickly deduce the
	existence of a quadric surface of revolution $\,S\,$ such that (modulo some fixed affine
	isomorphism)  $\,F(M_{A,b})=S\,$ for all $\,b<\infty\,$. Further, here as in the doubly-finite case, 
	the maximality of $\,A\,$ must be dictated by a critical point at height $\,A\,$. No cylinder has
	such a critical point, and among the quadrics, only the elliptic paraboloid and convex hyperboloid do.
	Clearly then, $\,S\,$ is one of these two surfaces, and $\,F(M)=S\,$.
	\end{proof}
\goodbreak

\section{Application to skew loops}\label{sec:skew}

We originally conceived our Main Theorem \ref{thm:main} above as a tool for proving the
existence of skewloops on a class of negatively curved tubes. In this final section we 
implement that idea.

\begin{definition}
	A \textbf{skewloop} is a circle differentiably immersed into $\,\R^{3}\,$ with \textbf{no}
	pair of parallel tangent lines.
\end{definition}

The existence of skewloops is not so obvious: Segre published the first construction
in 1968 \cite{segre}. A more recent construction and application appeared in M.~Ghomi's paper 
\cite{gh}, and sparked our own interest. We coined the term \emph{skewloop} in \cite{gs}, a 
subsequent joint paper that characterized \emph{positively} curved quadrics in $\,\R^{3}\,$ as the
\emph{only} surfaces having a point of positive curvature, but \textbf{no} skewloop:

\begin{thm}[{\cite[2002]{gs}}]
	A connected $\,C^{2}\,$ surface immersed in $\,\R^{3}\,$ with at least one point of 
	positive Gauss curvature admits no skewloop if and only if it is quadric.
\end{thm}

In particular, this identifies ellipsoids as the only \emph{compact} surfaces lacking
skewloops in $\,\R^{3}$. Its proof made strong use of Blaschke's result (Proposition 
\ref{prop:blaschke}) which, as explained in \S\ref{sec:intro}, applies to \emph{convex} surfaces 
only, and is fundamentally local.

Our dependence on Blaschke's theorem in \cite{gs} thus compelled us to assume positive curvature,
and at that time, we could only raise the question as to whether our skewloop-free characterization 
of quadrics might extend to non-positively curved surfaces \cite[Appendix B]{gs}.

S.~Tabachnikov, however, took a significant and interesting step toward an answer in \cite{tab},
when he showed that---modulo genericity and $\,C^{2}\,$ assumptions that were later eliminated in
\cite{ss}---\emph{negatively} curved quadrics admit no skewloops. That still left the converse
question open, however: Does lack of skewloops \emph{characterize} negatively curved quadrics?

We can now affirm that within a large class of surfaces, it does. To do so, we merely
combine results of the present paper with a lemma from \cite{gs}:

\begin{lem}[{\cite[Lemma 5.1]{gs}}]\label{lem:gs}
 	Suppose a $\,C^{2}\,$ embedded surface in $\,\R^{3}\,$ contains no skewloop,
	and some affine plane cuts it transversely along an oval $\,\ov\,$. Then 
	$\,\ov\,$ is central.
\end{lem}

Indeed, suppose $\,F:M\to\R^{3}\,$ immerses an open $\,C^{2}\,$ surface so that
it cuts some affine plane transversally along an oval $\,\ov\,$. Then $\,F\,$ clearly embeds some
annular neighborhood of $\,F^{-1}(\ov)\subset M\,$ into $\,\R^{3}\,$ as a transversely convex
tube. Such a tube either does, or does not, have \cpo, and correspondingly,
it either belongs to a central cylinder or quadric by Proposition \ref{prop:collar},
or else it contains a skewloop by Lemma \ref{lem:gs}. We have thus proven

\begin{prop}\label{prop:noskew}
	Suppose a $\,C^{2}$-immersed surface $\,M\subset\R^{3}\,$ cuts an affine plane
	transversally along an oval $\,\ov$, but admits no skewloop. Then some
	neighborhood of $\,\ov\,$ in $\,M\,$ belongs to a central cylinder or quadric.
\end{prop}

If we assume completeness, we get a more elegant global statement:

\begin{thm}\label{thm:noskew}
	Suppose a $C^{2}$-immersed surface $\,M\subset\R^{3}\,$ crosses some plane transversally 
	along an oval. Then exactly one of the following holds:	
	\begin{itemize}
		
		\item[(i)] $S\,$ contains a skewloop.
		\smallskip
		
		\item[(ii)] $S\,$ is the cylinder over an oval.
		\smallskip
		
		\item[(iii)] $S\,$ is a non-cylindrical quadric.
	\end{itemize} 
\end{thm}

\begin{proof}
	Our hypotheses explicitly guarantee the existence of at least one oval $\,\ov\,$
	along which $\,M\,$ cuts an affine plane transversally. But they actually ensure that \emph{all} 
	such ovals are central. For otherwise, Lemma \ref{lem:gs} puts a skewloop on $\,M\,$. 
	It follows that $\,M\,$ has \cpo, and the desired conclusion
	then follows from our Main Theorem \ref{thm:main}
\end{proof}

\begin{cor}
	Every complete embedded negatively curved surface that meets a plane transversely
	along an oval admits a skewloop, \emph{unless} it is affinely congruent to the tube 
	hyperboloid $\,x^{2}+y^{2}-z^{2}=1\,$.
\end{cor}

\begin{proof}
	This follows immediately from Theorem \ref{thm:noskew}, for among all cylinders
	and quadrics having a compact cross-section, only the tube hyperboloid has negative curvature.
\end{proof}



\section*{Acknowledgments} Many thanks to the Technion---Israel Institute of Technology---for their 
hospitality during a sabbatical in which much of this work got done, and to the Lady Davis Foundation 
and Indiana University for the financial support that made our visit there possible.



\vfill

\end{document}